\documentclass[12pt]{article}

\usepackage[T1]{fontenc}
\usepackage[utf8]{inputenc}
\usepackage[english]{babel}
\usepackage{amsmath,amsfonts,amsthm,amssymb,MnSymbol}

\usepackage{hyperref,url}
\usepackage{bussproofs}
\usepackage{xcolor}

\usepackage{tikz}
\usetikzlibrary{tikzmark}

\EnableBpAbbreviations

\newtheorem{thm}{Theorem}
\newtheorem{lem}{Lemma}
\newtheorem{prop}{Proposition}
\newtheorem{cor}{Corollary}
\theoremstyle{definition}
\newtheorem{definition}{Definition}
\newtheorem{remark}{Remark}
\newtheorem{example}{Example}

\author{Daniyar Shamkanov\\ \normalsize{\textit{Steklov Mathematical Institute of Russian Academy
of Sciences}}\\
\normalsize{\textit{Gubkina str. 8, 119991, Moscow, Russia}}\\
\normalsize{\textit{National Research University Higher School of Economics}}\\
\normalsize{\textit{Moscow, Russia}}\\ \normalsize{\textit{daniyar.shamkanov@gmail.com}}\\
}
\date{}

\title{(Non-)well-founded derivations in the provability logic $\mathsf{GLP}$\thanks{The article was prepared within the framework of the project ``International academic cooperation'' HSE University.}}
\begin{document}
\maketitle

  \begin{abstract}

We examine cyclic, non-well-founded and well-founded derivations in the provability logic $\mathsf{GLP}$. While allowing cyclic derivations does not change the system, the non-well-founded and well-founded derivations we consider define the same proper infinitary extension of $\mathsf{GLP}$. We establish that this extension is strongly algebraic and neighbourhood complete with respect to both local and global semantic consequence relations. In fact, these completeness results are proved for generalizations of global and local consequence relations, which we call global-local. In addition, we prove strong local neighbourhood completeness for the original system $\mathsf{GLP}$ (with ordinary derivations only).\\\\
\textit{Keywords:} provability logic, cyclic and infinitary derivations, algebraic semantics, neighbourhood semantics, 
local and global consequence relations. 
  \end{abstract}

\section{Introduction}

\label{s0}


A special and interesting type of infinitary derivations are non-well-founded derivations, or $\infty$-derivations for short. They are trees of formulas (sequents) constructed according to the inference rules of a deductive system, in which certain infinite branches are allowed. Notice that now widely known cyclic derivations can be understood as a particular case of non-well-founded ones.
In the given article, we focus on cyclic and non-well-founded derivations in an important provability logic $\mathsf{GLP}$. 

Recall that $\mathsf{GLP}$, introduced by Japaridze in  \cite{Japaridze1986}, is a modal propositional logic whose language contains countably many modal connectives $\Box_0$, $\Box_1$, $\Box_2$, etc. The connective $\Box_n$ can be understood as the predicate ``... is provable in Peano arithmetic extended with all true $\Pi^0_n$-sentences''. Being sound and complete with respect to the given provability semantics, $\mathsf{GLP}$ has deep applications in proof theory, especially, in ordinal analysis of arithmetic \cite{Beklemishev2005}. 

The fragment of $\mathsf{GLP}$ in the language with a single modal connective $\Box_0$ is known as the G\"{o}del-L\"{o}b provability logic $\mathsf{GL}$. A proof-theoretic presentation of $\mathsf{GL}$ in a form of a sequent calculus allowing non-well-founded proofs was given in \cite{Shamkanov2014, Iemhoff2016}. Non-well-founded derivations in the axiomatic calculus for $\mathsf{GL}$ were investigated in articles \cite{Shamkanov2017} and \cite{Shamkanov2020}. In the  given article, we continue this line of research and examine non-well-founded derivations in the axiomatic calculus for $\mathsf{GLP}$.


First, we show that allowing cyclic derivations in $\mathsf{GLP}$ does not change the derivability relation of the system. Then we observe that  non-well-founded derivations are closely related in $\mathsf{GLP}$ to another type of infinitary derivations, namely $\omega$-derivations. While $\infty$-derivations of $\mathsf{GLP}$ can be non-well-founded and have finite branching, $\omega$-derivations are well-founded and allow countable branching. We show that $\mathsf{GLP}$ extended with $\infty$-derivations is equivalent to $\mathsf{GLP}$ with certain $\omega$-derivations. In other words, two families of infinitary derivations in $\mathsf{GLP}$ define the same extension of the system. After this observation, we focus on algebraic and neighbourhood semantics of the extension.







Let us recall that neighbourhood semantics, independently developed by Scott \cite{Scott1970} and Montague \cite{Montague1970}, is a natural generalization of the relational one. 
Although $\mathsf{GLP}$ is incomplete with respect to its relational interpretation, Beklemishev and Gabelaia showed that it is weakly neighborhood complete (see \cite{Beklemishev2013}). In the given article, we strengthen this completeness result by considering local, global and global-local neighbourhood consequence relations over neighbourhood $\mathsf{GLP}$-frames. Accordingly, we introduce local, global and global-local derivability relations in the extension of $\mathsf{GLP}$ with infinitary derivations.

Let us remind the reader that, over neighbourhood $\mathsf{GLP}$-models, a formula $\varphi$ is a local semantic consequence of $\Gamma$ if, for any neighbourhood $\mathsf{GLP}$-model $\mathcal{M} $ and any world $x$ of $\mathcal{M}$,
\[(\forall \psi \in \Gamma \;\; \mathcal{M}, x \vDash \psi) \Rightarrow \mathcal{M}, x \vDash \varphi.\]
A formula $\varphi$ is a global semantic consequence of $\Gamma$ if, for any neighbourhood $\mathsf{GLP}$-model $\mathcal{M} $,
\[(\forall \psi \in \Gamma \;\;\mathcal{M} \vDash \psi) \Rightarrow \mathcal{M} \vDash \varphi.\]
In addition, $\varphi$ is a global-local consequence of sets of formulas $\Sigma$ and $\Gamma$ if
\[( (\forall \psi \in \Gamma \;\; \mathcal{M}, x \vDash \psi)  \wedge (\forall y \neq x \;\;\forall \xi \in \Sigma \;\; \mathcal{M}, y \vDash \xi )) \Longrightarrow \mathcal{M}, x \vDash \varphi\]
for any $\mathsf{GLP}$-model $\mathcal{M}$ and any world $x$ of $\mathcal{M}$. Trivially, the last consequence relation generalizes the local and the global ones. 

For the G\"{o}del-L\"{o}b provability logic $\mathsf{GL}$ with non-well-founded derivations, strong neighbourhood completeness was established in \cite{Shamkanov2017, Shamkanov2020} with respect to local, global and global-local consequence relations. In the given article, we obtain an analogous result for the provability logic $\mathsf{GLP}$.   As a consequence, we also prove strong local neighbourhood completeness of the original system $\mathsf{GLP}$ (with ordinary derivations only).

Finally, we note that this article, in which we summarize our knowledge of non-well-founded derivations in $\mathsf{GLP}$, is an extended version of a conference article \cite{Shamkanov2020a}. 

\section{Cyclic derivations in $\mathsf{GLP}$}
\label{s1}

{In this section, we recall the provability logic $\mathsf{GLP}$ and define a global-local derivability relation for this system, which generalizes standard local and global ones. In addition, we consider cyclic derivations in $\mathsf{GLP}$ and show that the given global-local derivablity relation is not affected if cyclic derivations are allowed. }

The provability logic $\mathsf{GLP}$ is a propositional modal logic in a language with infinitely many modal connectives $\Box_0, \Box_1, \dotsc$. In other words, formulas of the logic are built from the countable set of variables $\mathit{PV} = \{p, q, \dotsc\}$ and the constant $\bot$ using propositional connectives $\to$ and $\Box_i$ for each $i \in \mathbb{N}$. 
Other Boolean connectives and the modal connectives $\Diamond_i$ are considered as abbreviations:
\begin{gather*}
\neg \varphi := \varphi\to \bot,\qquad\top := \neg \bot,\qquad \varphi\wedge \psi := \neg (\varphi\to \neg \psi),
\\
\varphi\vee \psi := \neg \varphi\to \psi, \qquad \varphi \leftrightarrow \psi:=(\varphi\to \psi)\wedge (\psi \to \varphi),\qquad\Diamond_i \varphi := \neg\Box_i \neg \varphi.
\end{gather*}
By $\mathit{Fm}$, we denote the set of formulas of $\mathsf{GLP}$.


The provability logic $\mathsf{GLP}$ is defined by the following Frege-Hilbert calculus.\medskip

\textit{Axioms}:
\begin{itemize}
\item[(i)] the tautologies of classical propositional logic;
\item[(ii)] $\Box_i (\varphi\to \psi)\to (\Box_i \varphi\to \Box_i \psi)$;
\item[(iii)] $ \Box_i ( \Box_i \varphi\to \varphi)\to \Box_i \varphi$;
\item[(iv)] $\Diamond_i \varphi \to \Box_{i+1} \Diamond_i \varphi $;
\item[(v)] $\Box_i \varphi \to \Box_{i+1} \varphi $.
\end{itemize}

\textit{Inference rules}:
\begin{gather*}
\AXC{$\varphi$}
\AXC{$\varphi\to \psi$}
\LeftLabel{$\mathsf{mp}$}
\RightLabel{ ,}
\BIC{$\psi$}
\DisplayProof\qquad
\AXC{$\varphi$}
\LeftLabel{$\mathsf{nec}$}
\RightLabel{ .}
\UIC{$\Box_0 \varphi$}
\DisplayProof 
\end{gather*}
We remark that transitivity of the modal connectives $\Box_i$ is provable in $\mathsf{GLP}$, i.e. $\mathsf{GLP}\vdash \Box_i \psi \to \Box_i \Box_i \psi $ for any formula $\psi$ and any $i \in \mathbb{N}$.

A \emph{cyclic derivation} is a pair $(\kappa , d)$, where $\kappa$ is a finite tree of formulas of $\mathsf{GLP}$ constructed according to the inference rules ($\mathsf{mp}$) and ($\mathsf{nec}$), and $d$ is a function satisfying the following conditions:
the function $d$ is defined at some leaves of $\kappa$; the image $d(a)$ of a leaf $a$ lies on the path from the root of $\kappa$ to the leaf $a$ and is not equal to $a$; the path from $d(a)$ to $a$ intersects an application of the rule ($\mathsf{nec}$); $a$ and $d(a)$ are marked by the same formulas.
If the function $d$ is defined at a leaf $a$, then we say that nodes $a$ and $d(a)$ are connected by a \textit{back-link}. 

\begin{example}\label{example1} Consider the following cyclic derivation of a formula $\varphi$:
\begin{gather*}
\AXC{\tikzmark{E} $\varphi$}
\LeftLabel{\textsf{nec}} 
\UIC{$\Box_0  \varphi$}
\AXC{$ \Box_0 \varphi \rightarrow \varphi$}
\LeftLabel{\textsf{mp}}
\RightLabel{ .}
\BIC{\tikzmark{F} $\varphi$}
\DisplayProof 
\begin{tikzpicture}[overlay, remember picture, >=latex, distance=-3.0cm]
    \draw[->, thick] (pic cs:E) to [out=-25,in=20] (pic cs:F);
 \end{tikzpicture}
\end{gather*} 
Note that there exists an application of the rule ($\mathsf{nec}$) between two nodes connected by the unique back-link.
\end{example}

An \emph{assumption leaf} of a cyclic derivation $\delta = (\kappa, d)$ is a leaf that is not marked by an axiom of $\mathsf{GLP}$ and is not connected by a back-link. An assumption leaf is called \emph{boxed (local)} if the path from the root of $\kappa$ to the given leaf intersects an applications of the rule ($\mathsf{nec}$) or contains a node of the form $d(b)$ for some leaf $b$ (if there are no applications of the rule ($\mathsf{nec}$) on the path from the root of $\kappa$ to the given leaf). 

\begin{definition}
We put $\Sigma ;\Gamma \vdash_\textsf{cycl} \varphi$ if there is a cyclic derivation $\delta$ with the root marked by $\varphi$ in which all boxed assumption leaves are marked by some elements of $\Sigma$ and all local assumption leaves are marked by some elements of $\Gamma$. We also set  $\Sigma ;\Gamma \vdash \varphi$ if  $\Sigma ;\Gamma \vdash_\textsf{cycl} \varphi$ and the corresponding cyclic derivation does not contain back-links.
\end{definition}

Note that $\Sigma ;\Gamma \vdash \varphi$ if and only if $\mathsf{GLP} \vdash \bigwedge \Gamma^\prime \wedge \Box_0 \bigwedge \Sigma^\prime \to \varphi$ for some finite subsets $\Gamma^\prime$ of $\Gamma$ and $\Sigma^\prime$ of $\Sigma$. Hence, standard local and global derivability relations for $\mathsf{GLP}$ can be obtained from the relation $\vdash$ as follows:
\[\Gamma  \vdash_l \varphi \Longleftrightarrow \emptyset ;\Gamma  \vdash \varphi,\qquad\qquad
\Gamma  \vdash_g \varphi \Longleftrightarrow \Gamma ;\Gamma  \vdash \varphi.
\]

Let us denote the set of all local assumption leaves of a 
cyclic derivation $\delta$ by $\mathit{LA}(\delta)$ and the set of all boxed assumption leaves of $\delta$ by $\mathit{BA}(\delta)$.
\begin{lem}
Suppose $\delta = (\kappa, d)$ is  a cyclic derivation of a formula $\varphi$. Then
\[\mathsf{GLP} \vdash \bigwedge \lbrace  \psi_a \mid a \in \mathit{LA}(\delta) \rbrace \wedge \bigwedge \lbrace \Box_0 \psi_a \mid a \in \textit{BA}(\delta) \rbrace \rightarrow \varphi ,\] 
where, for each leaf $a$, $\psi_a$ is the formula of $a$.
\end{lem}
\begin{proof}
Assume there is a cyclic derivation $\delta = (\kappa, d)$ with the root marked by $\varphi$. We prove the assertion of the lemma by induction on the sum of the number of nodes in $\delta$ and the number of back-links. 

Case 1. Suppose that there are no leaves of $\delta$ connected by back-links with the root. 

Subcase 1A. If $\delta$ consists only of one leaf, then the required assertion holds immediately. 

Subcase 1B. Suppose $\delta$ has the form
\begin{gather*}
\AXC{$\delta^\prime$}
\noLine
\UIC{\vdots}
\noLine
\UIC{$\eta$}
\AXC{$\delta^{\prime\prime}$}
\noLine
\UIC{\vdots}
\noLine
\UIC{$\eta \rightarrow \varphi$}
\LeftLabel{\textsf{mp}}
\RightLabel{ ,}
\BIC{$\varphi$}
\DisplayProof 
\end{gather*}
where $\delta^\prime$ and $\delta^{\prime\prime}$ are cyclic derivations of $\eta$ and $\eta \rightarrow \varphi$ respectively.
Applying the induction hypothesis for $\delta^\prime$ and $\delta^{\prime\prime}$, we obtain
\begin{gather*}
\mathsf{GLP} \vdash \bigwedge \lbrace  \psi_a \mid a \in \mathit{LA}(\delta^\prime) \rbrace \wedge \bigwedge \lbrace \Box_0 \psi_a \mid a \in \textit{BA}(\delta^\prime) \rbrace \rightarrow \eta  ,\\
\mathsf{GLP} \vdash \bigwedge \lbrace  \psi_a \mid a \in \mathit{LA}(\delta^{\prime\prime}) \rbrace \wedge \bigwedge \lbrace \Box_0 \psi_a \mid a \in \textit{BA}(\delta^{\prime\prime}) \rbrace \rightarrow (\eta \rightarrow \varphi) .
\end{gather*}
Since $\mathit{LA}(\delta) = \mathit{LA}(\delta^\prime) \cup \mathit{LA}(\delta^{\prime\prime})$ and $\textit{BA}(\delta) = \textit{BA}(\delta^\prime) \cup \textit{BA}(\delta^{\prime\prime})$, we obtain 
\[\mathsf{GLP} \vdash \bigwedge \lbrace  \psi_a \mid a \in \mathit{LA}(\delta) \rbrace \wedge \bigwedge \lbrace \Box_0 \psi_a \mid a \in \textit{BA}(\delta) \rbrace \rightarrow (\eta \wedge (\eta \rightarrow \varphi)) .
\]
Consequently, 
\[\mathsf{GLP} \vdash \bigwedge \lbrace  \psi_a \mid a \in \mathit{LA}(\delta) \rbrace \wedge \bigwedge \lbrace \Box_0 \psi_a \mid a \in \textit{BA}(\delta) \rbrace \rightarrow    \varphi .
\]
Subcase 1C. Suppose $\delta$ has the form
\begin{gather*} 
\AXC{$\delta^\prime$}
\noLine
\UIC{\vdots}
\noLine
\UIC{$\eta$}
\LeftLabel{\textsf{nec}}
\RightLabel{ ,}
\UIC{$\Box_0 \eta$}
\DisplayProof 
\end{gather*}
where $\Box_0 \eta =\varphi$ and $\delta^\prime$ is a cyclic derivations of $\eta$.
By the induction hypothesis for $\delta^\prime$, we have
\begin{gather*}
\mathsf{GLP} \vdash \bigwedge \lbrace  \psi_a \mid a \in \mathit{LA}(\delta^\prime) \rbrace \wedge \bigwedge \lbrace \Box_0 \psi_a \mid a \in \textit{BA}(\delta^\prime) \rbrace \rightarrow    \eta  .
\end{gather*}
Therefore,
\begin{gather*}
\mathsf{GLP} \vdash \bigwedge \lbrace  \Box_0\psi_a \mid a \in \mathit{LA}(\delta^\prime) \rbrace \wedge \bigwedge \lbrace \Box_0 \Box_0 \psi_a \mid a \in \textit{BA}(\delta^\prime) \rbrace \rightarrow \Box_0\eta ,\\
\mathsf{GLP} \vdash \bigwedge \lbrace  \Box_0\psi_a \mid a \in \mathit{LA}(\delta^\prime) \rbrace \wedge \bigwedge \lbrace \Box_0 \psi_a \mid a \in \textit{BA}(\delta^\prime) \rbrace \rightarrow \Box_0\eta .
\end{gather*}
Since $\textit{BA}(\delta)= \mathit{LA}(\delta^\prime) \cup \textit{BA}(\delta^\prime) $, we obtain  
\begin{gather*}
\mathsf{GLP} \vdash \bigwedge \lbrace \Box_0 \psi_a \mid a \in \textit{BA}(\delta) \rbrace \rightarrow\varphi .
\end{gather*}
Note that, in this subcase, $\textit{LA}(\delta)=\emptyset$.

Case 2. Suppose there is a leaf $b$ connected by a back-link with the root of $\delta$. Consider the cyclic derivation $\delta^\prime$ obtained from $\delta$ by erasing the back-link. Since the path from the root of $\delta$ to the node $b$ intersects an application of the rule (\textsf{nec}), we see $b\in \textit{BA}(\delta^\prime)$. 
By the induction hypothesis for $\delta^\prime$, we have
\[\mathsf{GLP} \vdash (\bigwedge \lbrace \psi_a \mid a \in \textit{LA}(\delta^\prime) \rbrace \wedge \bigwedge \lbrace \Box_0 \psi_a \mid a \in \textit{BA}(\delta^\prime)\setminus\{b\} \rbrace \wedge \Box_0 \varphi) \rightarrow \varphi .\]
Hence,  
\begin{gather}
\mathsf{GLP} \vdash \bigwedge \lbrace \psi_a \mid a \in \textit{LA}(\delta^\prime) \rbrace \wedge \bigwedge \lbrace \Box_0 \psi_a \mid a \in \textit{BA}(\delta^\prime)\setminus\{b\} \rbrace \rightarrow (\Box_0 \varphi \rightarrow \varphi), \label{form2} \\
\mathsf{GLP} \vdash \bigwedge \lbrace \Box_0\psi_a \mid a \in \textit{LA}(\delta^\prime) \rbrace \wedge \bigwedge \lbrace \Box_0\Box_0 \psi_a \mid a \in \textit{BA}(\delta^\prime)\setminus\{b\} \rbrace \rightarrow \Box_0 (\Box_0 \varphi \rightarrow \varphi), \nonumber \\
\mathsf{GLP} \vdash \bigwedge \lbrace \Box_0\psi_a \mid a \in \textit{LA}(\delta^\prime) \rbrace \wedge \bigwedge \lbrace \Box_0 \psi_a \mid a \in \textit{BA}(\delta^\prime)\setminus\{b\} \rbrace \rightarrow \Box_0 (\Box_0 \varphi \rightarrow \varphi) \nonumber .
\end{gather}
Since $\mathsf{GLP} \vdash \Box_0 (\Box_0 \varphi \rightarrow \varphi) \rightarrow \Box_0 \varphi$, we see
\begin{gather} \label{form3} 
\mathsf{GLP} \vdash \bigwedge \lbrace \Box_0\psi_a \mid a \in \textit{LA}(\delta^\prime) \rbrace \wedge \bigwedge \lbrace \Box_0 \psi_a \mid a \in \textit{BA}(\delta^\prime)\setminus\{b\} \rbrace \rightarrow \Box_0 \varphi .
\end{gather}
From (\ref{form3}) and (\ref{form2}), we obtain
\[\mathsf{GLP} \vdash \bigwedge \lbrace \boxdot_0\psi_a \mid a \in \textit{LA}(\delta^\prime) \rbrace \wedge \bigwedge \lbrace \Box_0 \psi_a \mid a \in \textit{BA}(\delta^\prime)\setminus\{b\} \rbrace \rightarrow \varphi,\]
where $\boxdot_0 \psi\coloneq \psi \wedge \Box_0 \psi$. 

Since the leaf $b$ is connected by a back-link with the root of $\delta$, the set $\textit{BA}(\delta)$ consists of all assumption leaves of $\delta$. Therefore, $\textit{LA} (\delta^\prime)\cup ( \textit{BA}(\delta^\prime)\setminus \{b\}) \subset \textit{BA}(\delta)$. Also,  $\textit{LA}(\delta^\prime)=\textit{LA}(\delta)$. We conclude that 
\[\mathsf{GLP} \vdash \bigwedge \lbrace  \psi_a \mid a \in \mathit{LA}(\delta) \rbrace \wedge \bigwedge \lbrace \Box_0 \psi_a \mid a \in \textit{BA}(\delta) \rbrace \rightarrow    \varphi .\]
\end{proof}

\begin{thm}
For any sets of formulas $\Sigma$ and $\Gamma$, and any formula $\varphi$, we have 
\[\Sigma ;\Gamma \vdash \varphi \Longleftrightarrow \Sigma
 ;\Gamma \vdash_\mathsf{cycl} \varphi.\]
\end{thm}
\begin{proof}
The left-to-right implication holds immediately. We shall prove the converse.

Assume $ \Sigma
 ;\Gamma\vdash_\textsf{cycl} \varphi$. Then there is a cyclic derivation $\delta$ of a formula $\varphi$ such that $\mathit{LA}(\delta)\subset \Gamma$ and $\mathit{BA}(\delta)\subset \Sigma$. By the previous lemma, we have
\[\mathsf{GLP} \vdash \bigwedge \lbrace \psi_a \mid a \in \mathit{LA}(\delta) \rbrace \wedge \bigwedge \lbrace \Box_0 \psi_a \mid a \in \textit{BA}(\delta) \rbrace \rightarrow \varphi.\]
Trivially, 
\[\Sigma; \Gamma \vdash  \bigwedge \lbrace \psi_a \mid a \in \mathit{LA}(\delta) \rbrace \wedge \bigwedge \lbrace \Box_0 \psi_a \mid a \in \textit{BA}(\delta) \rbrace.\]
Consequently, $\Sigma; \Gamma \vdash \varphi$.
\end{proof}





\section{Infinitary derivations in $\mathsf{GLP}$}
\label{s1.5}
{In this section, we introduce two families of infinitary derivations in the Frege-Hilbert calculus for $\mathsf{GLP}$, $\infty$- and $\omega$-derivations. We prove that the derivability relations defined by these two families are equal.} 
 
A \emph{non-well-founded derivation,} or an \emph{$\infty$-derivation,} is a (possibly infinite) tree whose nodes are marked by formulas of $\mathsf{GLP}$ and that is constructed according to the rules ($\mathsf{mp}$) and ($\mathsf{nec}$). In addition, any infinite branch in an $\infty$-derivation must contain infinitely many applications of the rule ($\mathsf{nec}$). 

\begin{example}\label{example2} Consider the following $\infty$-derivation of a formula $\varphi_0$:
\begin{gather*}
\AXC{$\vdots$}
\noLine
\UIC{$\varphi_3$}
\LeftLabel{\textsf{nec}} 
\UIC{$\Box_0  \varphi_3$}
\AXC{$ \Box_0 \varphi_3 \rightarrow \varphi_2$}
\LeftLabel{\textsf{mp}}
\BIC{$\varphi_2$}
\LeftLabel{\textsf{nec}} 
\UIC{$\Box_0  \varphi_2$}
\AXC{$ \Box_0 \varphi_2 \rightarrow \varphi_1$}
\LeftLabel{\textsf{mp}}
\BIC{$\varphi_1$}
\LeftLabel{\textsf{nec}} 
\UIC{$\Box_0  \varphi_1$}
\AXC{$ \Box_0 \varphi_1 \rightarrow \varphi_0$}
\LeftLabel{\textsf{mp}}
\RightLabel{ .}
\BIC{$\varphi_0$}
\DisplayProof 
\end{gather*} 
Note that the unique infinite branch contains infinitely many applications of the rule ($\mathsf{nec}$).
\end{example}

Cyclic derivations considered in the {previous section} can be understood as a special case of $\infty$-derivations. The exact connection is as follows. An $\infty$-derivation is called \emph{regular} if it contains only finitely many non-isomorphic subtrees.
The unravelling of a cyclic derivation is obviously a regular $\infty$-derivation. {In addition, it is easy to prove the converse.}
\begin{prop}
Any regular $\infty$-derivation can be obtained by unravelling of a cyclic one.
\end{prop}
\begin{proof}
Assume we have a regular $\infty$-derivation $\delta$. Note that every node $a$ of the tree $\delta$ determines the subtree $\delta_a$ with the root $a$. Let $m$ denote the number of non-isomorphic subtrees
of $\delta$. Consider any {simple path} $a_0, a_1, \dotsc , a_m$ in $\delta$ that starts at the root of $\delta$ and has length $m $. This path defines the sequence of subtrees $\delta_{a_0},\delta_{a_1}, \dotsc,\delta_{a_m}$. Since $\delta$ contains precisely $m$ non-isomorphic subtrees, the path contains a pair of different nodes $b$ and $c$ determining isomorphic subtrees $\delta_b$ and $\delta_c$. Without loss of generality, assume that $c$ is farther from the root than $b$. Note that the path from $b$ to $c$ intersects an application of the rule ($\mathsf{nec}$) since otherwise there is an infinite branch in $\delta$ violating the condition on infinite branches of $\infty$-derivations. We cut the path $a_0, a_1, \dotsc , a_m$ at the node $c$ and connect $c$, which has become a leaf, with $b$ by a back-link. By applying a similar operation to each of the remaining paths of length $m $ that start at the root, we ravel $\delta$ into the desired cyclic derivation.

\end{proof}

Let us introduce another family of infinitary derivations in the Frege-Hilbert calculus for $\mathsf{GLP}$. An \emph{$\omega$-derivation} is a well-founded tree of formulas of $\mathsf{GLP}$ that is constructed according to the inference rules ($\mathsf{mp}$), ($\mathsf{nec}$) and the following rule:
\[
\AXC{$\Box_0\varphi_{1} \rightarrow \varphi_0$}
\AXC{$\Box_0\varphi_{2} \rightarrow \varphi_1$}
\AXC{$\Box_0\varphi_{3} \rightarrow \varphi_2$}
\AXC{$\dots$}
\LeftLabel{$\omega$}
\RightLabel{ .}
\QIC{$ \varphi_0$}
\DisplayProof 
\] 
In this inference, all premises except the leftmost one are called \emph{boxed}. 

An \emph{assumption leaf} of an $\infty$-derivation ($\omega$-derivation) is a leaf that is not marked by an axiom of $\mathsf{GLP}$. 

\begin{definition}\label{def2}
We set $\Gamma \Vdash_\infty \varphi$ ($\Gamma \Vdash_\omega \varphi$) if there is an $\infty$-derivation ($\omega$-derivation) with the root marked by $\varphi$ in which all assumption leaves are marked by some elements of $\Gamma$. 
\end{definition}

The \emph{main fragment} of an $\infty$-derivation ($\omega$-derivation) is a finite tree obtained from this derivation by cutting every branch at the nearest to the root premise of the rule ($\mathsf{nec}$) (at the nearest to the root premise of the rule ($\mathsf{nec}$) or boxed premise of the rule ($\omega$)). 
The \emph{local height $\lvert \delta \rvert$ of an $\infty$-derivation ($\omega$-derivation) $\delta$} is the length of the longest branch in its main fragment. An $\infty$-derivation ($\omega$-derivation) consisting of a single node has height $0$.

Note that the local height of the $\infty$-derivation from Example \ref{example2} equals to $1$ and its main fragment has the form
\begin{gather*}
\qquad\qquad
\AXC{$\Box_0 \varphi_1\qquad\qquad$}
\AXC{$\Box_0 \varphi_1 \to \varphi_0$}
\LeftLabel{$\mathsf{mp}$}
\RightLabel{ .}
\BIC{$\varphi_0$}
\DisplayProof 
\end{gather*}
The main fragment of the $\omega$-derivation consisting of a single application of the rule ($\omega$) has the form
\[
\AXC{$\Box_0\varphi_{1} \rightarrow \varphi_0$}
\AXC{$\qquad$}
\AXC{$\qquad$}
\AXC{$\qquad\qquad$}
\RightLabel{ ,}
\QIC{$ \varphi_0$}
\DisplayProof 
\] 
and its local height is also equal to $1$.
\begin{lem}\label{global inf <-> global omega}
For any set of formulas $\Gamma$ and any formula $\varphi$, we have 
\[\Gamma \Vdash_{\infty} \varphi \Longleftrightarrow \Gamma \Vdash_{\omega} \varphi .\]
\end{lem}
\begin{proof}
The right-to-left implication follows from the fact that $ \{\Box_0\varphi_{n+1} \rightarrow \varphi_n\mid n\in \mathbb{N}\} \Vdash_\infty \varphi_0$. Indeed, the formula $\varphi_0$ is derivable from $\{\Box_0\varphi_{n+1} \rightarrow \varphi_n \mid n\in \mathbb{N}\}$, which is shown in Example \ref{example2}. We prove the converse. 

Assume $\delta$ is an $\infty$-derivation with the root marked by $\varphi$ in which all assumption leaves are marked by some elements of $\Gamma$.
Recall that, for any node $a$ of the $\infty$-derivation $\delta$,  $\delta_a$ is the subtree of $\delta$ with the root $a$. We put $r(a)\coloneq\lvert \delta_a \rvert$. In addition, we denote the formula of the node $a$ by $\psi_a$. We say that a node $a$ belongs to the \emph{$(n+1)$-th slice of $\delta$} if there are precisely $n $ applications of the rule ($\mathsf{nec}$) on the path from the node $a$ to the root of $\delta$. By $\xi_n$, we denote the formula $\bigwedge \{\psi_a \mid \text{$a$ belongs to the $(n+1)$-th slice of $\delta$} \}$.
 
We claim that $\Gamma \Vdash_\omega \Box_0 \xi_{n+1} \rightarrow \xi_n$ for any $n \in \mathbb{N}$. It is sufficient to prove that $\Gamma \Vdash_\omega \Box_0 \xi_{n+1} \rightarrow \psi_a$ whenever $a$ belongs to the $(n+1)$-th slice of $\delta$. The proof is by induction on $r(a)$.

If $\psi_a$ is an axiom of $\mathsf{GLP}$ or an element of $\Gamma$, then we immediately obtain the required statement.
Otherwise, $\psi_a$ is obtained by an application of an inference rule in $\delta$.

Case 1. If $\psi_a$ is obtained by the rule ($\mathsf{nec}$), then this formula has the form $\Box_0 \psi_{b}$, where $b$ is the premise of $a$. We see that $b$ belongs to the $(n+2)$-th slice of $\delta$. Consequently, $\Gamma \Vdash_\omega \xi_{n+1}\rightarrow  \psi_b$ and $\Gamma \Vdash_\omega \Box_0 \xi_{n+1} \rightarrow \psi_a$. 

Case 2. Suppose $\psi_a$ is obtained by the rule ($\mathsf{mp}$). Consider the premises $b_1$ and $b_2$ of $a$. We have $r(b_1)< r(a)$ and $r(b_2)< r(a)$. From the induction hypothesis, we obtain $\Gamma \Vdash_\omega \Box_0 \xi_{n+1} \rightarrow (\psi_{b_1} \wedge \psi_{b_2})$. Since $\psi_{b_2} = \psi_{b_1} \rightarrow \psi_a$, we have $\Gamma \Vdash_\omega  \Box_0 \xi_{n+1} \rightarrow \psi_a$.

This proves the claim that $\Gamma \Vdash_\omega \Box_0 \xi_{n+1} \rightarrow \xi_n$ for any $n \in \mathbb{N}$. Applying ($\omega$), we obtain $\Gamma \Vdash_\omega \xi_0$. Besides, $\Gamma \Vdash_\omega \xi_0 \rightarrow \varphi$. We conclude that $\Gamma \Vdash_\omega \varphi$.
\end{proof}

{An assumption leaf of an $\infty$-derivation ($\omega$-derivation) is called \emph{boxed} if the path from the given leaf to the root of the tree intersects an application of the rule ($\mathsf{nec}$) (the path intersects an application of the rule ($\mathsf{nec}$) or an application of the rule ($\omega$) on a boxed
premise). }


\begin{definition}
We set $\Sigma; \Gamma \Vdash_\infty \varphi$ ($\Sigma; \Gamma \Vdash_\omega \varphi$) if there is an $\infty$-derivation ($\omega$-derivation) with the root marked by $\varphi$ in which all boxed assumption leaves are marked by some elements of $\Sigma$ and all non-boxed assumption leaves are marked by some elements of $\Gamma$. 
\end{definition}
Let us note that this definition gives us more general relations than Definition  \ref{def2}:
\[
\Gamma  \Vdash_\infty \varphi \Longleftrightarrow \Gamma ;\Gamma  \Vdash_\infty \varphi,\qquad\qquad
\Gamma  \Vdash_\omega \varphi \Longleftrightarrow \Gamma ;\Gamma  \Vdash_\omega \varphi.
\]

\begin{thm} \label{inf <-> omega}
For any sets of formulas $\Sigma$ and $\Gamma$, and any formula $\varphi$, we have 
\[\Sigma ;\Gamma \Vdash_\infty \varphi \Longleftrightarrow \Sigma ;\Gamma \Vdash_\omega \varphi .\]
\end{thm}
\begin{proof}
We begin with the left-to-right implication. Assume $\delta$ is an $\infty$-derivation with the root marked by $\varphi$ in which all boxed assumption leaves are marked by some elements of $\Sigma$ and all non-boxed assumption leaves are marked by some elements of $\Gamma$. By induction on $\lvert \delta \rvert $, we show that 
$\Sigma ;\Gamma \Vdash_\omega \varphi$.

If $\varphi$ is an axiom of $\mathsf{GLP}$ or an element of $\Gamma$, then we obtain the required statement immediately.
Otherwise, consider the lowermost application of an inference rule in $\delta$. 

Case 1. Suppose that $\delta$ has the form
\begin{gather*}
\AXC{$\delta^\prime$}
\noLine
\UIC{$\vdots$}
\noLine
\UIC{$\psi$}
\AXC{$\delta^{\prime\prime}$}
\noLine
\UIC{$\vdots$}
\noLine
\UIC{$\psi \to \varphi$}
\LeftLabel{$\mathsf{mp}$}
\RightLabel{ .}
\BIC{$\varphi$}
\DisplayProof
\end{gather*}
By the induction hypothesis applied to $ \delta^\prime $ and $\delta^{\prime\prime} $, we have 
$\Sigma ;\Gamma \Vdash_\omega \psi\to\varphi$ and $\Sigma ;\Gamma \Vdash_\omega \psi$. Therefore, $\Sigma ;\Gamma \Vdash_\omega \varphi$.

Case 2. Suppose that $\delta$ has the form
\begin{gather*}
\AXC{$\delta^\prime$}
\noLine
\UIC{$\vdots$}
\noLine
\UIC{$\psi$}
\LeftLabel{$\mathsf{nec}$}
\RightLabel{ ,}
\UIC{$\Box_0 \psi$}
\DisplayProof
\end{gather*}
where $\Box_0 \psi =\varphi$. We see that $\Sigma \Vdash_\infty \psi$. 
By Lemma \ref{global inf <-> global omega}, we have $\Sigma \Vdash_\omega \psi$. Applying the rule ($\mathsf{nec}$), we obtain $\Sigma ; \emptyset \Vdash_\omega \varphi$ and $\Sigma ; \Gamma \Vdash_\omega \varphi$.

Now we check the right-to-left implication. Assume $\delta$ is an $\omega$-derivation with the root marked by $\varphi$ in which all boxed assumption leaves are marked by some elements of $\Sigma$ and all non-boxed assumption leaves are marked by some elements of $\Gamma$. By induction on $\lvert \delta \rvert $, we prove that 
$\Sigma ;\Gamma \Vdash_\infty \varphi$.

Let us consider only the main case when $\delta$ has the form
\begin{gather*}
\AXC{$\delta^\prime$}
\noLine
\UIC{$\vdots$}
\noLine
\UIC{$\Box_0\varphi_{1} \rightarrow \varphi_0$}
\AXC{$\delta^{\prime\prime}$}
\noLine
\UIC{$\vdots$}
\noLine
\UIC{$\Box_0\varphi_{2} \rightarrow \varphi_1$}
\AXC{$\delta^{\prime\prime\prime}$}
\noLine
\UIC{$\vdots$}
\noLine
\UIC{$\Box_0\varphi_{3} \rightarrow \varphi_2$}
\AXC{$\dots$}
\LeftLabel{$\omega$}
\RightLabel{ ,}
\QIC{$ \varphi_0$}
\DisplayProof 
\end{gather*}
where $\varphi_0=\varphi$.
We see that $\Sigma \Vdash_\omega \Box_0\varphi_{n+2} \rightarrow \varphi_{n+1}$ for every $n\in \mathbb{N}$. From Lemma \ref{global inf <-> global omega}, we have $\Sigma \Vdash_\infty \Box_0\varphi_{n+2} \rightarrow \varphi_{n+1}$. By the induction hypothesis applied to $ \delta^\prime $, we also have $\Sigma ;\Gamma \Vdash_\infty \Box_0\varphi_{1} \rightarrow \varphi_0$. Hence, we obtain
\begin{gather*}
\AXC{$\vdots$}
\noLine
\UIC{$\varphi_3$}
\LeftLabel{\textsf{nec}} 
\UIC{$\Box_0  \varphi_3$}
\AXC{$\sigma_2$}
\noLine
\UIC{$\vdots$}
\noLine
\UIC{$ \Box_0 \varphi_3 \rightarrow \varphi_2$}
\LeftLabel{\textsf{mp}}
\BIC{$\varphi_2$}
\LeftLabel{\textsf{nec}} 
\UIC{$\Box_0  \varphi_2$}
\AXC{$\sigma_1$}
\noLine
\UIC{$\vdots$}
\noLine
\UIC{$ \Box_0 \varphi_2 \rightarrow \varphi_1$}
\LeftLabel{\textsf{mp}}
\BIC{$\varphi_1$}
\LeftLabel{\textsf{nec}} 
\UIC{$\Box_0  \varphi_1$}
\AXC{$\sigma_0$}
\noLine
\UIC{$\vdots$}
\noLine
\UIC{$ \Box_0 \varphi_1 \rightarrow \varphi_0$}
\LeftLabel{\textsf{mp}}
\RightLabel{ ,}
\BIC{$\varphi_0$}
\DisplayProof 
\end{gather*} 
where $\sigma_0$ is the corresponding $\infty$-derivtion in which all boxed assumption leaves are marked by some elements of $\Sigma$ and all non-boxed assumption leaves are marked by some elements of $\Gamma$, and $\sigma_n$, for $n>0$, are $\infty$-derivtions in which all assumption leaves are marked by elements of $\Sigma$. We conclude that $\Sigma ;\Gamma \Vdash_\infty \varphi$.
\end{proof}

In what follows, we will write $\Sigma ;\Gamma \Vdash \varphi$ instead of $\Sigma ;\Gamma \Vdash_\omega \varphi$ and $\Gamma  \Vdash_{ g} \varphi$ instead of $\Gamma \Vdash_\omega \varphi$. {The relation $\emptyset ;\Gamma \Vdash_\omega \varphi$ is also denoted by $\Gamma  \Vdash_{ l} \varphi$.

Trivially,
\[\Sigma ;\Gamma \vdash \varphi \Longrightarrow \Sigma ;\Gamma \Vdash \varphi.\]
The converse implication will be proved in the final section for the case when $\Sigma$ is finite.}


\section{Algebraic semantics}
\label{s2}

In this section, we consider algebraic semantics of the provability logic $\mathsf{GLP}$ extended with infinitary derivations.

A \emph{Magari algebra} (or a \emph{diagonalizable algebra}) 
$\mathcal{A}= ( A, \wedge, \vee, \to, 0, 1, \Box)$ is a Boolean algebra $( A, \wedge, \vee, \to, 0, 1)$ together with an {operation} $\Box \colon A \to A$ satisfying the identities:
 \[ \Box 1=1 , \quad  \Box (x \wedge y) = \Box x \wedge \Box y, \quad  \Box(\Box x\to x) =\Box x .\]

For any Magari algebra $\mathcal{A}$, the operation $\Box$ is monotone with respect to the order (of the Boolean part) of $\mathcal{A}$. Indeed, if $a \leqslant b$, then $a \wedge b =a$, $\Box a \wedge \Box b =\Box (a \wedge b) = \Box a$, and $\Box a \leqslant \Box b$. In addition, we remark that an inequality $\Box x \leqslant \Box \Box x$ holds in any Magari algebra.


We call a Magari algebra \emph{$\Box$-founded (or Pakhomov-Walsh-founded)\footnote{This notion was inspired by an article of Pakhomov and Walsh \cite{Pakhomov2021}.}} if, for every sequence of its elements $( a_i)_{i\in \mathbb{N}}$ such that $\Box a_{i+1}\leqslant a_{i}$, we have $a_0=1$.
Note that, for any such sequence $( a_i)_{i\in \mathbb{N}}$, all elements $a_i$ are equal to $1$ in any $\Box$-founded Magari algebra.

We give a series of examples of $\Box$-founded Magari algebras. A Magari algebra is called \emph{$\sigma$-complete} if its underlying Boolean algebra is $\sigma$-complete, that is, any countable subset $S$ of this algebra has the least upper bound $\bigvee S$. An equivalent condition is that every countable subset $S$ has the greatest lower bound $\bigwedge S$. 
\begin{prop}\label{from sigma-completness to box-foundness}
Any $\sigma$-complete Magari algebra is $\Box$-founded.\footnote{This statement was inspired by a correspondence with Tadeusz Litak (see also the proof of Theorem 2.15 from \cite{Litak2005}).} 
\end{prop}
  
\begin{proof}
Assume we have a $\sigma$-complete Magari algebra $\mathcal{A}$ and a sequence of its elements $( a_i)_{i\in \mathbb{N}}$ such that $\Box a_{i+1}\leqslant a_{i}$. We shall prove that $a_0=1$. 

Put $b= \bigwedge\limits_{i \in \mathbb{N}} a_i$. For any $i\in \mathbb{N}$, we have $b \leqslant  a_{i+1}$ and $\Box b \leqslant\Box a_{i+1}\leqslant a_{i}$. Hence, 
\[\Box b \leqslant b, \qquad \Box b \to b =1, \qquad \Box b = \Box (\Box b \to b) =\Box 1 =1, \qquad b=1 . \] We obtain that $a_0=1$.
\end{proof}

\begin{remark} {Let us additionally mention an arithmetical example of $\Box$-founded Magari algebra without going into details. If we consider the second-order arithmetical theory $\mathsf{ACA}_0$ extended with all true $\Sigma^1_1$-sentences, then its provability algebra forms a $\Box$-founded Magari algebra. This observation can be obtained following the lines of Theorem 3.2 from \cite{Pakhomov2021}.}
\end{remark}

The notion of $\Box$-founded Magari algebra $\mathcal{A}$ can be also defined in terms of the binary relation $\prec_\mathcal{A}$ on $\mathcal{A}$: 
\[a\prec_\mathcal{A} b \Longleftrightarrow \Box a \leqslant b.\]
\begin{prop}[see Proposition 3.1 from \cite{Shamkanov2020}]
For any Magari algebra $\mathcal{A}=( A, \wedge, \vee, \to, 0, 1, \Box )$, the relation $\prec_\mathcal{A}$ is a strict partial order on $A\setminus \{1\}$.
\end{prop}
\begin{prop}[see Proposition 3.2 from \cite{Shamkanov2020}]\label{box-foundness is equivalent to well-foundness}
For any Magari algebra $\mathcal{A}=( A, \wedge, \vee, \to, 0, 1, \Box )$, the algebra $\mathcal{A}$ is $\Box$-founded if and only if the partial order $\prec_\mathcal{A}$ on $A\setminus \{1\}$ is well-founded.
\end{prop} 


A Boolean algebra $( A, \wedge, \vee, \to, 0, 1)$ together with a sequence of unary {operations} $\Box_0, \Box_1, \dotsc$ is called a \emph{$\mathsf{GLP}$-algebra} if, for each $i \in \mathbb{N}$, it satisfies the following conditions:
\begin{enumerate}
\item $( A, \wedge, \vee, \to, 0, 1, \Box_i )$ is a Magari algebra;
\item $\Diamond_i a \leqslant \Box_{i+1} \Diamond_i a $ for any $a \in A$;
\item $\Box_i a \leqslant \Box_{i+1} a $ for any $a \in A$.  
\end{enumerate}
 
For a $\mathsf{GLP}$-algebra $\mathcal{A}= (A, \wedge, \vee, \to, 0, 1, \Box_0, \Box_1, \dotsc)$, the Magari algebra $(A, \wedge, \vee, \to, 0, 1, \Box_i)$ is denoted by $\mathcal{A}_i$. A $\mathsf{GLP}$-algebra $\mathcal{A}$ is called \emph{$\Box$-founded ($\sigma$-complete)} if the Magari algebra $\mathcal{A}_0 $ is $\Box$-founded ($\sigma$-complete). From Proposition \ref{from sigma-completness to box-foundness}, we immediately see that any $\sigma$-complete $\mathsf{GLP}$-algebra is $\Box$-founded. In addition, it can be easily shown that $\mathcal{A}_i$ is $\Box$-founded for every $i\in \mathbb{N}$ whenever $\mathcal{A}$ is $\Box$-founded.

Now we define semantic consequence relations over $\Box$-founded $\mathsf{GLP}$-algebras corresponding
to derivability relations $\Vdash_{l}$, $\Vdash_{g}$ and $\Vdash$. A \emph{valuation in a $\mathsf{GLP}$-algebra $\mathcal{A}=( A, \wedge, \vee, \to, 0, 1, \Box_0, \Box_1, \dotsc )$} is standardly defined as a function $v \colon \mathit{Fm} \to A$ such that $v (\bot) = 0$,  $v (\varphi \to \psi) =  v (\varphi) \to v(\psi)$ and $v (\Box_i \varphi) = \Box_i v (\varphi)$. For a subset $S$ of a $\mathsf{GLP}$-algebra $\mathcal{A}$, the filter of (the Boolean reduct of) $\mathcal{A}$ generated by $S$ is denoted by $\langle S \rangle$.


\begin{definition}
Given a set of formulas $\Gamma$ and a formula $\varphi$, we set $\Gamma \VDash_l \varphi$ if for any $\Box$-founded $\mathsf{GLP}$-algebra $\mathcal{A}$ and any valuation $v$ in $\mathcal{A}$ 
\[ v(\varphi)\in \langle \{v(\psi)\mid \psi \in \Gamma\}\rangle . \]
We also set $\Gamma \VDash_g \varphi$ if for any $\Box$-founded $\mathsf{GLP}$-algebra $\mathcal{A}$ and any valuation $v$ in $\mathcal{A}$ 
\[ (\forall \psi \in \Gamma\;\; v(\psi)=1) \Rightarrow v(\varphi)=1 . \] 
In addition, we set $\Sigma ; \Gamma \VDash \varphi$ if for any $\Box$-founded $\mathsf{GLP}$-algebra $\mathcal{A}$ and any valuation $v$ in $\mathcal{A}$
\[ (\forall \xi \in \Sigma\;\; \Box_0 v(\xi)=1) \Longrightarrow v(\varphi)\in \langle \{v(\psi)\mid \psi \in \Gamma\}\rangle . \]
\end{definition}

The relation $\VDash$ is a generalization of $\VDash_l$ and $\VDash_g $ since $\emptyset ; \Gamma \VDash \varphi \Leftrightarrow \Gamma \VDash_l \varphi$ and $\Gamma ;\Gamma  \VDash \varphi \Leftrightarrow \Gamma  \VDash_g \varphi$.
The only nontrivial implication is the following.



\begin{lem}\label{algebraic lemma}
For any set of formulas $\Gamma$ and any formula $\varphi$, we have
\[\Gamma \VDash_g \varphi \Longrightarrow  \Gamma ;\Gamma\VDash \varphi.\]
\end{lem}
\begin{proof}
Assume $\Gamma \VDash_g \varphi$. In addition, assume we have a $\Box$-founded $\mathsf{GLP}$-algebra $\mathcal{A}=(A, \wedge, \vee, \to, 0, 1, \Box_0, \Box_1, \dotsc )$ together with a valuation $v$ in $\mathcal{A}$ such that $\Box_0 v(\psi)=1$ for any $\psi \in \Gamma$. We shall prove that $v(\varphi) \in \langle \{ v(\psi) \mid \psi \in \Gamma \} \rangle$.

We denote the filter $\langle \{ v(\psi) \mid \psi \in \Gamma \} \rangle$ of $\mathcal{A}$ by $F$. Let us check that $F$ is an open filter that is $\Box_i a \in F$ for every $i\in \mathbb{N}$ whenever $a\in F$. 
If $a\in F$, then $v(\psi_1)\wedge\dotsb \wedge v(\psi_k) \leqslant a$ for a finite set of formulas $\{\psi_1, \dotsc, \psi_k\}\subset \Gamma$. 
Consequently, $\Box_0 v(\psi_1)\wedge\dotsb \wedge\Box_0\psi(\psi_k) \leqslant \Box_0 a$. Since $\Box_0 v(\psi)=1$ for any $\psi \in \Gamma$, we obtain $\Box_0 a =1$ and $\Box_0 a \in F$. Notice that $\Box_0 a \leqslant \Box_i a$ for each $i\in \mathbb{N}$. Hence, for all $i\in \mathbb{N}$, $\Box_i a \in F$.

Now the quotient $\mathsf{GLP}$-algebra $\mathcal{A}/ F$ and the canonical epimorphism $f \colon \mathcal{A} \to \mathcal{A} / F$ are well-defined. We claim that the algebra $\mathcal{A}/ F$ is $\Box$-founded. Assume there exists a sequence of elements $(a_i)_{i\in \mathbb{N}}$ of $\mathcal{A}$ such that $\Box_0 f(a_{i+1})\leqslant f(a_{i})$. We see that $f(\Box_0 a_{i+1} \to a_i)=1$ and $(\Box_0 a_{i+1} \rightarrow a_i) \in F$. Therefore, there exists a sequence $(\Gamma_i)_{i\in \mathbb{N}}$ of finite subsets of $\Gamma$ such that $\bigwedge \{  v(\psi) \mid \psi\in \Gamma_i\}  \leqslant (\Box_0 a_{i+1} \to a_i)$ in $\mathcal{A}$.
Consequently, $\bigwedge \{  v(\psi) \mid \psi\in \Gamma_i\}  \wedge \Box_0 a_{i+1} \leqslant  a_i$  and $ \bigwedge \{  \Box_0 v(\psi) \mid \psi\in \Gamma_i\}  \wedge \Box_0 \Box_0 a_{i+1} \leqslant  \Box_0 a_i$.  
Since $\Box_0 v(\psi)=1$ for any $\psi \in \Gamma$, we have $\Box_0 \Box_0 a_{i+1} \leqslant  \Box_0 a_i$ in $\mathcal{A}$. From $\Box$-foundedness of $\mathcal{A}$, we obtain $\Box_0 a_i =1$ for any $i\in \mathbb{N}$. Since $\bigwedge \{  v(\psi) \mid \psi\in \Gamma_i\}  \wedge \Box_0 a_{i+1} \leqslant  a_i$, we obtain $\bigwedge \{  v(\psi) \mid \psi\in \Gamma_i\}  \leqslant  a_i$. Thus, $a_i \in F$ for any $i\in \mathbb{N}$ and $ f (a_0) =1$. The algebra $\mathcal{A}/ F$ is $\Box$-founded.

Let us consider the valuation $f \circ v$ in $\mathcal{A}/ F$. We see that $(f\circ v)(\psi)=1$ in $\mathcal{A}/ F$ for every $\psi \in \Gamma$. From the assumption $\Gamma \VDash_g \varphi$, it follows that $(f\circ v)(\varphi)=1$. Hence, $v(\varphi) \in F =\langle \{ v(\psi) \mid \psi \in \Gamma \} \rangle$.\end{proof}

\begin{lem}\label{algebraic soundness}
For any set of formulas $\Gamma$ and any formula $\varphi$, we have
\[\Gamma \Vdash_{g} \varphi \Longrightarrow  \Gamma \VDash_g \varphi.\]
\end{lem}
\begin{proof}

The lemma is easily obtained by transfinite induction on the ordinal heights of $\omega$-derivations. The main case of an inference rule ($\omega$) holds since $ \{\Box_0\varphi_{n+1} \rightarrow \varphi_n\mid n\in \mathbb{N}\} \VDash_g \varphi_0$, which follows from the definition of $\Box$-founded $\mathsf{GLP}$-algebra.


\end{proof}

\begin{prop}[algebraic soundness]\label{algebraic soundness 2}
For any sets of formulas $\Sigma$ and $\Gamma$, and any formula $\varphi$, we have
\[\Sigma;\Gamma \Vdash \varphi \Longrightarrow  \Sigma;\Gamma \VDash \varphi.\]
\end{prop}
\begin{proof}
Assume $\Sigma;\Gamma \Vdash \varphi$. In addition, assume we have a $\Box$-founded $\mathsf{GLP}$-algebra $\mathcal{A}$ together with a valuation $v$ such that $\Box_0 v(\xi)=1$ for every $\xi \in \Sigma$. Applying Theorem \ref{inf <-> omega}, we find an $\infty$-derivation $\delta$ with the root marked by $\varphi$ in which all boxed assumption leafs are marked by some elements of $\Sigma$ and all non-boxed assumption leafs are marked by some elements of $\Gamma$. By induction on $\lvert \delta \rvert $, we prove that 
$v(\varphi) \in \langle \{ v(\psi) \mid \psi \in \Gamma \} \rangle$.

If $\varphi$ is an axiom of $\mathsf{GLP}$ or an element of $\Gamma$, then we obtain the required statement immediately.
Otherwise, consider the lowermost application of an inference rule in $\delta$. 

Case 1. Suppose that $\delta$ has the form
\begin{gather*}
\AXC{$\delta^\prime$}
\noLine
\UIC{$\vdots$}
\noLine
\UIC{$\eta$}
\AXC{$\delta^{\prime\prime}$}
\noLine
\UIC{$\vdots$}
\noLine
\UIC{$\eta \to \varphi$}
\LeftLabel{$\mathsf{mp}$}
\RightLabel{ .}
\BIC{$A$}
\DisplayProof
\end{gather*}
By the induction hypotheses for $ \delta^\prime $ and $\delta^{\prime\prime} $, we have 
$v(\eta) \in \langle \{ v(\psi) \mid \psi \in \Gamma \} \rangle$ and $v(\eta\to \varphi) \in \langle \{ v(\psi) \mid \psi \in \Gamma \} \rangle$. Therefore, $v(\eta) \wedge v(\eta\to \varphi) \in \langle \{ v(\psi) \mid \psi \in \Gamma \} \rangle$. Since $v(\eta) \wedge v(\eta\to \varphi) \leqslant v(\psi)$, we obtain
$v(\varphi) \in \langle \{ v(\psi) \mid \psi \in \Gamma \} \rangle$.

Case 2. Suppose that $\delta$ has the form
\begin{gather*}
\AXC{$\delta^\prime$}
\noLine
\UIC{$\vdots$}
\noLine
\UIC{$\eta$}
\LeftLabel{$\mathsf{nec}$}
\RightLabel{ ,}
\UIC{$\Box_0 \eta$}
\DisplayProof
\end{gather*}
where $\Box_0 \eta =\varphi$. By Lemma \ref{global inf <-> global omega}, we have $\Sigma \Vdash_g \eta$. 
From Lemma \ref{algebraic soundness}, we see that $\Sigma \VDash_g  \eta$. Applying Lemma \ref{algebraic lemma}, we also have $\Sigma, \Sigma \VDash  \eta$. Therefore, $v(\eta) \in \langle \{ v(\xi) \mid \xi \in \Sigma \} \rangle$ and there exists a finite subset $\Sigma_0$ of $\Sigma$ such that $\bigwedge \{ v(\xi) \mid \xi \in \Sigma_0 \} \leqslant v(\eta)$. Consequently, $\bigwedge \{ \Box_0 v(\xi) \mid \xi \in \Sigma_0 \} \leqslant \Box_0 v(\eta)$. We obtain $v(\varphi)= \Box_0 v(\eta) =1 $ from the assumption $\Box_0v(\xi)=1$ for every $\xi \in \Sigma$. We conclude that $v(\varphi) \in \langle \{ v(\psi) \mid \psi \in \Gamma \} \rangle$.
\end{proof}

\begin{thm}[algebraic completeness]
\label{algebraic completeness}
For any sets of formulas $\Sigma$ and $\Gamma$, and any formula $\varphi$, we have
\[\Sigma;\Gamma \Vdash \varphi \Longleftrightarrow  \Sigma;\Gamma \VDash \varphi.\]
\end{thm}
\begin{proof}
The left-to-right implication follows from Proposition \ref{algebraic soundness 2}. We prove the converse. Assume $\Sigma ; \Gamma \VDash \varphi$. Consider the theory $T=\{ \theta \in \mathit{Fm} \mid \Sigma ; \emptyset \Vdash \theta\}$. We see that $T$ contains all axioms of $\mathsf{GLP}$ and is closed under the rules ($\mathsf{mp}$) and ($\mathsf{nec}$). We define an equivalence relation $\sim_T$ on the set of formulas $Fm$ by putting $\mu \sim_T \rho$ if and only of $(\mu \leftrightarrow \rho) \in T$. Let us denote the equivalence class of $\mu$ by $[\mu]_T$. Applying the Lindenbaum-Tarski construction, we obtain a $\mathsf{GLP}$-algebra $\mathcal{L}_T$ on the set of equivalence classes of formulas, where $[\mu]_T\wedge [\rho]_T = [\mu\wedge \rho]_T$, $[\mu]_T \vee [\rho]_T = [\mu\vee \rho]_T$, $[\mu]_T\to [\rho]_T = [\mu\to \rho]_T$, $0 = [\bot]_T$, $1= [\top]_T$ and $ \Box_i [\mu]=[\Box_i \mu]$.


Let us check that the algebra $\mathcal{L}_T$ is $\Box$-founded. Assume we have a sequence of formulas $(\mu_i)_{i\in \mathbb{N}}$ such that $\Box_0 [\mu_{i+1}]_T\leqslant [\mu_{i}]_T$. We have $[\Box_0 \mu_{i+1} \to \mu_i]_T=1$ and $(\Box_0 \mu_{i+1} \to \mu_i) \in T$. For every $i\in \mathbb{N}$, there exists an $\omega$-derivation $\delta_i$ for the formula $\Box_0 \mu_{i+1} \rightarrow \mu_i$ such that all assumption leaves of $\delta_i$ are boxed and marked by some elements of $\Sigma$. We obtain the following $\omega$-derivation for the formula $\mu_0$:
\begin{gather*}
\AXC{$\delta_0$}
\noLine
\UIC{$\vdots$}
\noLine
\UIC{$\Box_0\mu_{1} \rightarrow \mu_0$}
\AXC{$\delta_1$}
\noLine
\UIC{$\vdots$}
\noLine
\UIC{$\Box_0\mu_{2} \rightarrow \mu_1$}
\AXC{$\delta_2$}
\noLine
\UIC{$\vdots$}
\noLine
\UIC{$\Box_0\mu_{3} \rightarrow \mu_2$}
\AXC{$\dots$}
\LeftLabel{$\omega$}
\RightLabel{ ,}
\QIC{$ \mu_0$}
\DisplayProof 
\end{gather*}
where all assumption leaves are boxed and marked by some elements of $\Sigma$. Hence, $\mu_0 \in T$ and $[\mu_0]_T =[\top]_T= 1$. We conclude that the $\mathsf{GLP}$-algebra $\mathcal{L}_T$ is $\Box$-founded.

Consider the valuation $v \colon \theta \mapsto [\theta]_T$ in the $\mathsf{GLP}$-algebra $\mathcal{L}_T$. Since $\{\Box_0 \xi \mid \xi \in \Sigma\} \subset T$, we have $\Box_0 v(\xi)=1$ for any $\xi \in \Sigma$. From the assumption $ \Sigma ; \Gamma \VDash \varphi$, we obtain that $v(\varphi) \in \langle \{v(\psi)\mid \psi \in \Gamma\}\rangle$. Consequently, there is a finite subset $\Gamma_0$ of $\Gamma$ such that $\bigwedge \{v(\psi)\mid \psi \in \Gamma_0\} \leqslant v(\varphi)$ in $\mathcal{L}_T$. We have $(\bigwedge \{[\psi]_T\mid \psi \in \Gamma_0\} \to [\varphi]_T) =1$ and $(\bigwedge \Gamma_0 \to \varphi) \in T$, i.e. $\Sigma ; \emptyset \Vdash \bigwedge \Gamma_0 \to \varphi$. Notice that $\emptyset ; \Gamma \Vdash \bigwedge \Gamma_0$. Therefore, $\Sigma ; \Gamma \Vdash \bigwedge \Gamma_0$ and $\Sigma ; \Gamma \Vdash \Gamma_0 \to \varphi$. Applying an inference rule ($\mathsf{mp}$), we conclude $\Sigma ; \Gamma \Vdash  \varphi$.  

\end{proof}

\section{Neighbourhood semantics}
\label{s3}
{

In this section, we recall neighbourhood semantics of the provability logic $\mathsf{GLP}$ and consider local, global and global-local semantic consequence relations over the class of neighbourhood $\mathsf{GLP}$-frames. 
}


Recall that a point $x$ of a topological space is called \emph{isolated} if the set $\{x\}$ is open. A topological space $ ( X, \tau)$ is \emph{scattered} if each non-empty subset of $X$ (as a topological space with the inherited topology) has an isolated point.
 
A subset $U$ of a topological space $ ( X, \tau)$ is a \emph{neighbourhood} of $x\in X$ if $U$ is open and contains $x$. A topological space is $T_d$ if every point of the space is closed in some of its neighbourhoods (with respect to the inherited topology).
\begin{prop}[see Lemma 4.8 from \cite{Shamkanov2020}]\label{T_d-prop}
Any scattered topological space is $T_d$.  
\end{prop} 

An \emph{Esakia frame} (or a \emph{Magari frame}) $\mathcal{X} = ( X, \Box )$ is a set $X$ together with a mapping $\Box \:\colon \mathcal{P}(X) \to \mathcal{P}(X)$ such that the powerset Boolean algebra $\mathcal{P}(X)$ with the mapping $\Box$ forms a Magari algebra. 

We briefly recall a connection between scattered topological spaces and Esakia frames 
(cf. \cite{Beklemishev2014}). Note that we allow Esakia frames and topological spaces to be empty. 

A subset $U$ of a topological space $ ( X, \tau)$ is a \emph{punctured neighbourhood} of $x\in X$ if $x \nin U$ and $U\cup \{x \}$ is open.
For $V\subset X$, the \emph{derived set $d_\tau (V)$ of $V$}
is the set of limit points of $V$: 
$$x \in d_\tau (V) \Longleftrightarrow \forall U \in \tau \; (x\in U \Rightarrow \exists y \neq x \; (y \in U \cap V))\;.$$
The \emph{co-derived set $cd_\tau(V)$ of $V$} is defined as $X \setminus d_\tau (X \setminus V)$. 
By definition, $x \in cd_\tau(V)$ if and only if there exists a punctured neighbourhood of $x$ entirely contained in $V$. Notice that $V$ is open if and only if $V \subset cd_\tau(V)$.

\begin{prop}[{Simmons \cite{Simmons1975}, Esakia \cite{Esakia1981}}]\label{from scattered spaces to Esakia frames}
If $( X, \tau)$ is a scattered topological space, then $( X, \mathit{cd}_\tau)$ is an Esakia frame.
\end{prop}

\begin{prop}[{Esakia \cite{Esakia1981}}]
If $( X, \Box)$ is an Esakia frame, then $X$ bears a unique topology $\tau$ for which $\Box = \mathit{cd}_\tau$. Moreover, the space $( X,\tau)$ is scattered.
\end{prop}


A \emph{$\mathsf{GLP}$-space} is a multitopological space $( X, \tau_0, \tau_1, \dotsc) $, where, for each $i \in \mathbb{N}$, $\tau_i$ is scattered, $\tau_i \subset \tau_{i+1}$, and $\mathit{d}_{\tau_i}(V)\in\tau_{i+1}$ for any $V\in \mathcal{P}(X) $.

A \emph{neighbourhood $\mathsf{GLP}$-frame}  $\mathcal{X} = ( X, \Box_0, \Box_1, \dotsc )$ is a set $X$ together with a sequence of unary operations $\Box_0, \Box_1, \dotsc$ on $\mathcal{P}(X)$ such that the powerset Boolean algebra $\mathcal{P}(X)$ with the given operations forms a $\mathsf{GLP}$-algebra. Elements of $X$ are called \emph{worlds} of the frame $\mathcal{X} $. A \emph{neighbourhood $\mathsf{GLP}$-model} is a pair $\mathcal {M}= ( \mathcal{X}, v )$, where $\mathcal{X}$ is a neighbourhood $\mathsf{GLP}$-frame and $v$ is a valuation in the powerset $\mathsf{GLP}$-algebra of $\mathcal{X}$.
A formula $\varphi$ is \emph{true at a world $x$ of a model $\mathcal{M}$}, written as $\mathcal{M}, x \vDash \varphi$, if $x \in v(\varphi)$. 
A formula $\varphi$ is called \emph{true in $\mathcal{M}$}, written as $\mathcal{M} \vDash \varphi$, if $\varphi$ is true at all worlds of $\mathcal{M}$. 
 
\begin{prop}[{see Proposition 4 from \cite{Beklemishev2014}}] \text{ } \label{GLP-spaces=GLP-frames}
\begin{enumerate}
\item If $( X, \tau_0, \tau_1, \dotsc)$ is a $\mathsf{GLP}$-space, then $( X, \mathit{cd}_{\tau_0}, \mathit{cd}_{\tau_1}, \dotsc)$ is a $\mathsf{GLP}$-frame.
\item If $( X, \Box_0, \Box_1, \dotsc) $ is a $\mathsf{GLP}$-frame, then $X$ bears a unique series of topologies $\tau_0, \tau_1, \dotsc$ such that $\Box_i = \mathit{cd}_{\tau_i}$ for every $i\in \mathbb{N}$. Moreover, the multitopological space $( X, \tau_0, \tau_1, \dotsc)$ is a $\mathsf{GLP}$-space.
\end{enumerate}
\end{prop}

In what follows, we don't distinguish $\mathsf{GLP}$-frames and the corresponding multitopological spaces in such a way that we use topological notions related to $( X, \tau_0, \tau_1, \dotsc)$ for the frame $( X, \mathit{cd}_{\tau_0}, \mathit{cd}_{\tau_1}, \dotsc)$. For example, we say that a subset $U$ is \emph{$n$-open} in $( X, \Box_0, \Box_1, \dotsc) $ if it belongs to the corresponding $n$-th topology on $X$, which is equivalent to $U \subset \Box_n U$.

For a $0$-open subset $X^\prime$ of a $\mathsf{GLP}$-frame $\mathcal{X}=(X, \Box_0, \Box_1, \dotsc)$, the \emph{open subframe of $\mathcal{X}$ determined by $X^\prime$} is defined as $(X^\prime, \Box^\prime_0, \Box^\prime_1, \dotsc)$, where $\Box^\prime_i V = X^\prime \cap \Box_i V $ for any $V \subset X^\prime$ and any $i\in \mathbb{N}$.

\begin{lem}\label{subframe}
Any open subframe $\mathcal{X}^\prime$ of a $\mathsf{GLP}$-frame $\mathcal{X}$ is a $\mathsf{GLP}$-frame. Moreover, a subset $V$ of $\mathcal{X}^\prime$  is $i$-open if and only if it is $i$-open in $\mathcal{X}$.
\end{lem}
\begin{proof}
Assume we have a $\mathsf{GLP}$-frame $\mathcal{X}=(X, \Box_0, \Box_1, \dotsc)$ and a $0$-open subset $X^\prime$ of $\mathcal{X}$. 

Let us consider the mapping $f\colon \mathcal{P}(X)\to \mathcal{P}(X^\prime)$ such that $f(V)= X^\prime \cap V$ for any $V\in \mathcal{P}(X)$. For each $i\in \mathbb{N}$, we have
\[f(\Box_i V)= X^\prime \cap \Box_i V = X^\prime \cap  \Box_i X^\prime \cap \Box_i V= X^\prime \cap  \Box_i (X^\prime \cap V)= \Box^\prime_i f(V)\]
since $X^\prime \subset \Box_0 X^\prime \subset \Box_i X^\prime$. We see that $f$ is a surjective homomorphism from the powerset $\mathsf{GLP}$-algebra of $\mathcal{X}$ to the Boolean algebra $\mathcal{P}(X^\prime)$ expanded with the operations $\Box^\prime_0, \Box^\prime_1, \dotsc$. Hence, the latter algebra satisfies every identity satisfied by the former one. Therefore, the expansion of $\mathcal{P}(X^\prime)$ is a $\mathsf{GLP}$-algebra, and $(X^\prime, \Box^\prime_0, \Box^\prime_1, \dotsc)$ is a $\mathsf{GLP}$-frame.

Moreover, for any $V\subset X^\prime$, we have
\[V\subset \Box^\prime_i V \Longleftrightarrow V\subset X^\prime\cap \Box_i V \Longleftrightarrow V\subset \Box_i V. \]
Consequently, $V$ is $i$-open in $\mathcal{X}^\prime$ if and only if it is $i$-open in $\mathcal{X}$.

\end{proof}

\begin{lem}\label{submodel}
Suppose $(\mathcal{X}^\prime, v^\prime)$ and $(\mathcal{X}, v)$ are $\mathsf{GLP}$-models, where $\mathcal{X}^\prime$ is an open subframe of $\mathcal{X}$ and $v^\prime(p)=X^\prime \cap v(p)$ for any $p\in \mathit{PV}$. Then $v^\prime(\varphi) = X^\prime \cap v(\varphi)$ for any formula $\varphi$.
\end{lem}
\begin{proof}
This lemma is proved by induction on the structure of $\varphi$. Let us consider only the main case when $\varphi$ has the form $\Box_i \psi$. By the induction hypothesis, we have $v^\prime(\psi) = X^\prime \cap v(\psi)$. For any $i \in \mathbb{N}$, we obtain 
\begin{multline*} 
v^\prime(\varphi)= v^\prime(\Box_i \psi) = \Box^\prime_i v^\prime(\psi) = X^\prime \cap \Box_i v^\prime(\psi) = X^\prime \cap \Box_i (X^\prime \cap v(\psi)) = \\
X^\prime \cap \Box_i X^\prime \cap \Box_i v(\psi)= X^\prime \cap \Box_i v(  \psi)=  X^\prime \cap v( \Box_i \psi)=X^\prime \cap v( \varphi)
\end{multline*}
since $X^\prime \subset \Box_0 X^\prime \subset \Box_i X^\prime$.
\end{proof}

Let us recall the following neighbourhood completeness result obtained by Beklemishev and Gabelaia in \cite{Beklemishev2013}. We will establish stronger versions of this result in the final two sections.
 
\begin{thm}
For any formula $\varphi$, if $\mathsf{GLP} \nvdash \varphi $, then there is a $\mathsf{GLP}$-model $\mathcal{M}$ and a world $x$ of $\mathcal{M}$ such that  $\mathcal{M}, x \nvDash \varphi$.
\end{thm}

Now we define semantic consequence relations over neighbourhood $\mathsf{GLP}$-frames corresponding
to derivability relations $\Vdash_{l}$, $\Vdash_{g}$, $\Vdash$ and show some properties of these relations. 

\begin{definition}
Given a set of formulas $\Gamma$ and a formula $\varphi$, we set $\Gamma \vDash_l \varphi$ if, for any $\mathsf{GLP}$-model $\mathcal{M}$ and any world $x$ of $\mathcal{M}$,
\[(\forall \psi \in \Gamma \;\; \mathcal{M}, x \vDash \psi)   \Longrightarrow \mathcal{M}, x \vDash \varphi.\]
We also set $\Gamma \vDash_g \varphi$ if, for any $\mathsf{GLP}$-model $\mathcal{M}$,
\[(\forall \psi \in \Gamma \;\; \mathcal{M} \vDash \psi)   \Longrightarrow \mathcal{M} \vDash \varphi.\]
In addition, we set $\Sigma ; \Gamma \vDash \varphi$ if, for any $\mathsf{GLP}$-model $\mathcal{M}$ and any world $x$ of $\mathcal{M}$,
\[( (\forall \psi \in \Gamma \;\; \mathcal{M}, x \vDash \psi)  \wedge (\forall y \neq x \;\;\forall \xi \in \Sigma \;\; \mathcal{M}, y \vDash \xi )) \Longrightarrow \mathcal{M}, x \vDash \varphi.\]
\end{definition}

Trivially, the relation $\vDash$ is a generalization of $\vDash_l$ and $\vDash_g $ since $\emptyset ; \Gamma \vDash \varphi \Leftrightarrow \Gamma \vDash_l \varphi$ and $\Gamma ;\Gamma  \vDash \varphi \Leftrightarrow \Gamma  \vDash_g \varphi$.

\begin{prop}\label{neighbourhood soundness}
For any sets of formulas $\Sigma$ and $\Gamma$, and for any formula $\varphi$, we have
\[\Sigma; \Gamma \VDash \varphi \Longrightarrow \Sigma; \Gamma \vDash \varphi.\]
\end{prop}
\begin{proof}
{
Assume $\Sigma; \Gamma \VDash \varphi$. In addition, assume we have a $\mathsf{GLP}$-model $\mathcal{M}= ((X, \Box_0, \Box_1,\dotsc ), v)$ and a world $x$ of $\mathcal{M}$ such that $\mathcal{M}, x \vDash \psi$ for each $\psi \in \Gamma$ and
\[ \forall y \neq x \;\;\forall \xi \in \Sigma \;\; \mathcal{M}, y \vDash \xi .\] We shall prove that $\mathcal{M}, x \vDash \varphi$.

Let us denote the $\mathsf{GLP}$-frame $ (X, \Box_0, \Box_1,\dotsc )$ by $\mathcal{X}$. From Proposition \ref{T_d-prop}, there exists a $0$-neighbourhood $X^\prime$ of $x$ such that $X^\prime\setminus \{x\}$ is $0$-open. Consider the open subframe $(X^\prime, \Box^\prime_0, \Box^\prime_1, \dotsc)$ of $\mathcal{X}$ determined by $X^\prime $ and denote it by $\mathcal{X}^\prime$. Note that $\mathcal{X}^\prime$ is a $\mathsf{GLP}$-frame by Lemma \ref{subframe}. We define the valuation $v^\prime$ over the frame $\mathcal{X}^\prime$ such that $v^\prime(p)=X^\prime \cap v(p)$ for all $p\in \mathit{PV}$. From Lemma \ref{submodel}, for any formula $\eta$, $v^\prime (\eta)= X^\prime\cap v(\eta)$. Therefore, $\mathcal{M}^\prime , x \vDash \psi$ for all $\psi \in \Gamma$, where $\mathcal{M}^\prime = (\mathcal{X}^\prime, v^\prime)$. In addition, we have
\[ \forall y \in X^\prime\setminus \{x\}\;\;\forall \xi \in \Sigma \;\; \mathcal{M}^\prime, y \vDash \xi .\]
Note that $X^\prime\setminus \{x\}$ is $0$-open in $\mathcal{X}^\prime=(X^\prime, \Box^\prime_0, \Box^\prime_1, \dotsc)$ by Lemma \ref{subframe}. Hence, every point of $\mathcal{X}^\prime$ has a punctured $0$-neighbourhood entirely contained in $v^\prime (\xi)$ for each $\xi\in \Sigma$. We have $ \Box^\prime_0 v^\prime (\xi)= X^\prime$ for any $\xi\in \Sigma$. From the assumption $\Sigma; \Gamma \VDash \varphi$, it follows that
\[v^\prime(\varphi) \in \langle \{ v^\prime(\psi) \mid \psi \in \Gamma \} \rangle . \]
Consequently, there exists a finite subset $\Gamma^\prime$ of $\Gamma$ such that \[ \bigcap \{v^\prime(\psi) \mid \psi \in \Gamma^\prime\}\subset v^\prime(\varphi),\]
where 
\[\bigcap \emptyset\coloneq X^\prime.\] 
Now we see 
\[x \in \bigcap\{ v^\prime(\psi) \mid \psi \in \Gamma \} \subset \bigcap\{ v^\prime(\psi) \mid \psi \in \Gamma^\prime \}\subset v^\prime(\varphi).\]
Therefore, $\mathcal{M}^\prime, x \vDash \varphi$. Applying Lemma \ref{submodel}, we conclude that $\mathcal{M}, x \vDash \varphi$.
}
\end{proof}

The semantic consequence relation $\vDash$ can be equivalently defined in a more general form as follows. 

\begin{definition}
Given two sets of formulas $\Sigma$ and $\Gamma$, and a formula $\varphi$, we put $\Sigma; \Gamma \vDash^\ast \varphi$ if
\[( (\forall \psi \in \Gamma \;\; \mathcal{M}, x \vDash \psi)  \wedge (\forall y \in U \setminus \{x\} \;\;\forall \xi \in \Sigma \;\; \mathcal{M}, y \vDash \xi )) \Longrightarrow \mathcal{M}, x \vDash \varphi\]
for any $\mathsf{GLP}$-model $\mathcal{M}$, any world $x$ of $\mathcal{M}$ and any $0$-neighbourhood $U$ of $x$ in $\mathcal{M}$. 
\end{definition}

\begin{prop}\label{two glocal neighbourhood relations}
For any sets of formulas $\Sigma$ and $\Gamma$ and any formula $\varphi$, we have
\[\Sigma ;\Gamma \vDash \varphi \Longleftrightarrow  \Sigma ;\Gamma \vDash^\ast \varphi.\]
\end{prop}
\begin{proof}
{
The right-to-left implication trivially holds. Let us prove the converse. Assume $\Sigma ;\Gamma \vDash \varphi$. In addition, assume there is a $\mathsf{GLP}$-model $\mathcal{M}= ((X, \Box_0, \Box_1,\dotsc ), v)$, a world $x$ of $\mathcal{M}$ and a $0$-neighbourhood $X^\prime$ of $x$ such that $\mathcal{M}, x \vDash \psi$ for each $\psi \in \Gamma$ and
\[\forall y \in X^\prime \setminus \{x\} \;\;\forall \xi \in \Sigma \;\; \mathcal{M}, y \vDash \xi .\]
We shall prove $\mathcal{M}, x \vDash \varphi$.  

By $\mathcal{X}$, we denote the $\mathsf{GLP}$-frame $ (X, \Box_0, \Box_1,\dotsc )$. Consider the open subframe $(X^\prime, \Box^\prime_0, \Box^\prime_1, \dotsc)$ of the $\mathsf{GLP}$-frame $\mathcal{X}$ determined by $X^\prime $ and denote it by $\mathcal{X}^\prime$. Note that $\mathcal{X}^\prime$ is a $\mathsf{GLP}$-frame by Lemma \ref{subframe}. Furthermore, define the valuation $v^\prime$ over the frame $\mathcal{X}^\prime$ such that $v^\prime(p)=X^\prime \cap v(p)$ for all $p\in \mathit{PV}$. From Lemma \ref{submodel}, for any formula $\eta$, we have $v^\prime (\eta)= X^\prime\cap v(\eta)$. Therefore, $\mathcal{M}^\prime , x \vDash \psi$ for each $\psi \in \Gamma$, where $\mathcal{M}^\prime = (\mathcal{X}^\prime, v^\prime)$. In addition, 
\[ \forall y \in X^\prime\setminus \{x\}\;\;\forall \xi \in \Sigma \;\; \mathcal{M}^\prime, y \vDash \xi .\]
From the assumption $\Sigma; \Gamma \vDash \varphi$, it follows that
$\mathcal{M}^\prime , x \vDash \varphi$. Applying Lemma \ref{submodel}, we conclude $\mathcal{M}, x \vDash \varphi$.

}
\end{proof}

\section{Representation of $\Box$-founded Magari algebras}
\label{s4}
In this section, we prove that any $\Box$-founded Magari algebra can be embedded into the powerset Magari algebra of an Esakia frame. We also obtain some related results, which will be applied in the next section.

From Proposition \ref{box-foundness is equivalent to well-foundness}, we know that a Magari algebra $\mathcal{A}=( A, \wedge, \vee, \to, 0, 1, \Box )$ is $\Box$-founded if and only if the binary relation $\prec_\mathcal{A}$ is well-founded on $A\setminus \{1\}$, where 
\[a\prec_\mathcal{A} b \Longleftrightarrow \Box a \leqslant b.\]
Let us recall some basic properties of well-founded relations.
 
A \emph{well-founded set} is a pair $\mathcal{S} =( S, \prec )$, where $\prec$ is a well-founded relation on $S$. For any element $a $ of $\mathcal{S}$, its ordinal height in $\mathcal{S}$ is denoted by $\mathit{ht}_\mathcal{S}(a)$. Recall that $\mathit{ht}_\mathcal{S}$ is defined by transfinite recursion on $\prec$ as follows:
\[\mathit{ht}_\mathcal{S} (a)= \sup \{ \mathit{ht}_\mathcal{S} (b)+1 \mid b \prec a \}.\] 

A \emph{homomorphism} from $\mathcal{S}_1=( S_1, \prec_1 )$ to $\mathcal{S}_2=( S_2, \prec_2 )$ is a function $f \colon S_1 \to S_2$ such that $f(b)\prec_2 f(c)$ whenever $b \prec_1 c$. 
\begin{prop}\label{prop1}
Suppose $f\colon \mathcal{S}_1 \to \mathcal{S}_2$ is a homomorphism of well-founded sets and $a$ is an element of $\mathcal{S}_1$. Then $\mathit{ht}_{\mathcal{S}_1} (a) \leqslant\mathit{ht}_{\mathcal{S}_2} (f(a))$. 
\end{prop}

For well-founded sets $\mathcal{S}_1 =( S_1, \prec_1 )$ and $\mathcal{S}_2 =( S_2, \prec_2 )$, their product $\mathcal{S}_1\times \mathcal{S}_2$ is defined as the set $S_1 \times S_2$ together with the following relation
\[( b_1 , b_2 ) \prec ( c_1 , c_2 )  \Longleftrightarrow b_1 \prec_1 c_1 \text{ and } b_2 \prec_2 c_2. \] 
Clearly, $\prec$ is a well-founded relation on $S_1 \times S_2$.


\begin{prop}\label{prop2}
Suppose $a$ and $b$ are elements of well-founded sets $\mathcal{S}_1$ and $\mathcal{S}_2$ respectively. Then $\mathit{ht}_{\mathcal{S}_1 \times \mathcal{S}_2} (( a, b)) = \min \{ \mathit{ht}_{\mathcal{S}_1} (a), \mathit{ht}_{\mathcal{S}_2} (b) \}$. 
\end{prop}

For an element $a$ of a $\Box$-founded Magari algebra $\mathcal{A}$, define $\mathit{ht}_\mathcal{A}(a)$ as the ordinal height of $a$ with respect to $\prec_\mathcal{A}$. We put $\mathit{ht}_\mathcal{A}(a) = \infty$ if $a=1$.

\begin{lem}\label{basic}
Suppose $a$ and $b$ are elements of a $\Box$-founded Magari algebra $\mathcal{A}$. Then $\mathit{ht}_\mathcal{A} (a\wedge b) = \min \{\mathit{ht}_\mathcal{A} (a), \mathit{ht}_\mathcal{A} (b) \} $ and $\mathit{ht}_\mathcal{A} (a) +1 \leqslant \mathit{ht}_\mathcal{A} (\Box a)$, where we define $\infty+1 :=\infty$.\footnote{This lemma was inspired by a conversation with Fedor Pakhomov.}
\end{lem}
\begin{proof}
Assume we have a $\Box$-founded Magari algebra $\mathcal{A}=( A, \wedge, \vee, \to, 0, 1, \Box )$ and two elements $a$ and $b$ of $\mathcal{A}$.

First, we prove that $\mathit{ht}_\mathcal{A} (a\wedge b) = \min \{\mathit{ht}_\mathcal{A} (a), \mathit{ht}_\mathcal{A} (b) \} $. If $a=1$ or $b=1$, then the equality immediately holds. Suppose $a \neq 1$ and $b \neq 1$. Let $\mathcal{S}$ be the set $A\setminus \{1\}$ together with the well-founded relation $\prec_\mathcal{A}$. We have $a\wedge b \neq 1$, $\mathit{ht}_\mathcal{A} (a) = \mathit{ht}_\mathcal{S} (a)$, $\mathit{ht}_\mathcal{A} (b) = \mathit{ht}_\mathcal{S} (b)$ and $\mathit{ht}_\mathcal{A} (a\wedge b) = \mathit{ht}_\mathcal{S} (a\wedge b)$. The mapping
\[f \colon (c,d) \mapsto c\wedge d\] 
is a homomorphism from $\mathcal{S} \times \mathcal{S}$ to $\mathcal{S}$. From Proposition \ref{prop2} and Proposition \ref{prop1}, we have  
\[\min \{\mathit{ht}_\mathcal{S} (a), \mathit{ht}_\mathcal{S} (b) \}
= \mathit{ht}_{\mathcal{S}\times \mathcal{S}} ((a,b)) \leqslant \mathit{ht}_\mathcal{S} (a\wedge b).\]
Consequently, \[\min \{\mathit{ht}_\mathcal{A} (a), \mathit{ht}_\mathcal{A} (b) \} \leqslant \mathit{ht}_\mathcal{A} (a\wedge b).\]
On the other hand, $\mathit{ht}_\mathcal{A} (a\wedge b) \leqslant \mathit{ht}_\mathcal{A} (a)$ since 
\[\{ e \in A \setminus \{1\} \mid e \prec_\mathcal{A} (a\wedge b)\} \subset \{ e \in A \setminus \{1\} \mid e \prec_\mathcal{A} a\}.\]
Analogously, we have  $\mathit{ht}_\mathcal{A} (a\wedge b) \leqslant \mathit{ht}_\mathcal{A} (b)$. It follows that 
\[\mathit{ht}_\mathcal{A} (a\wedge b) = \min \{\mathit{ht}_\mathcal{A} (a), \mathit{ht}_\mathcal{A} (b) \} .\]

Now we prove $\mathit{ht}_\mathcal{A} (a) +1 \leqslant \mathit{ht}_\mathcal{A} (\Box a)$. If $\Box a=1$, then the inequality immediately holds. Suppose $\Box a \neq 1$. Then $a \neq 1$. We see $a \prec_\mathcal{A} \Box a$. The required inequality holds from the definition of $\mathit{ht}_\mathcal{A}$. 
\end{proof}

For a $\Box$-founded Magari algebra $\mathcal{A}=(A, \wedge, \vee, \to, 0, 1, \Box)$ and an ordinal $\gamma$, put $\mathit{M}_\mathcal{A}(\gamma)= \{ a \in A \mid \gamma \leqslant\mathit{ht}_\mathcal{A} (a)\} $.
We see that $\mathit{M}_\mathcal{A}(0)= A$ and $\mathit{M}_\mathcal{A}(\eta) \supset \mathit{M}_\mathcal{A}(\gamma)$ whenever $\eta \leqslant \gamma$. 

\begin{lem}
For any $\Box$-founded Magari algebra $\mathcal{A}$ and any ordinal $\gamma$, the set $\mathit{M}_\mathcal{A}(\gamma) $ is a filter in $\mathcal{A}$. 
\end{lem}
\begin{proof}
Suppose $a$ and $b$ belong to $\mathit{M}_\mathcal{A}(\gamma) $. Then $\gamma \leqslant \mathit{ht}_\mathcal{A} (a)$ and $\gamma \leqslant \mathit{ht}_\mathcal{A} (b)$. We have 
$\gamma \leqslant \min \{\mathit{ht}_\mathcal{A} (a), \mathit{ht}_\mathcal{A} (b) \} = \mathit{ht}_\mathcal{A} (a\wedge b)  $ by Lemma \ref{basic}. Consequently $a \wedge b$ belongs to $\mathit{M}_\mathcal{A}(\gamma) $.

Now suppose $c$  belongs to $\mathit{M}_\mathcal{A}(\gamma) $ and $c\leqslant d$. We shall show that $d \in\mathit{M}_\mathcal{A}(\gamma) $. We have $\gamma \leqslant \mathit{ht}_\mathcal{A} (c) = \mathit{ht}_\mathcal{A} (c\wedge d) = \min \{\mathit{ht}_\mathcal{A} (c), \mathit{ht}_\mathcal{A} (d) \} \leqslant \mathit{ht}_\mathcal{A} (d) $ by Lemma \ref{basic}. Hence $d \in\mathit{M}_\mathcal{A}(\gamma) $.  
\end{proof}

Let $\mathit{Ult}\: \mathcal{A}$ be the set of all ultrafilters of (the Boolean part of) a Magari algebra $\mathcal{A}=(A, \wedge, \vee, \to, 0, 1, \Box)$. Put $\widehat{a} = \{u \in \mathit{Ult}\: \mathcal{A} \mid a \in u\}$ for $a\in A$. We recall that the mapping $\:\widehat{\cdot}\;\colon a \mapsto \widehat{a}\:$ is an embedding of the Boolean algebra $(A, \wedge, \vee, \to, 0, 1)$ into the powerset Boolean algebra $\mathcal{P}(\mathit{Ult}\: \mathcal{A})$ by Stone's representation theorem.


\begin{lem}\label{important lemma}
Suppose $\mathcal{A}=(A, \wedge, \vee, \to, 0, 1, \Box)$  is a $\Box$-founded Magari algebra and $F$ is a filter of $\mathcal{A}$ such that $F\subset \Box^{-1} 1$, where $\Box^{-1} 1= \{a\in A \mid \Box a =1\}$. Then there exists a scattered topology $\tau$ on $\mathit{Ult}\: \mathcal{A}$ such that $\widehat{\Box a} = \mathit{cd}_\tau (\widehat{a})$ for any element $a$ of $\mathcal{A}$. Moreover, $\mathit{cd}_\tau (\bigcap \{\widehat{a}\mid a\in F\})= \mathit{Ult}\: \mathcal{A}$.
\end{lem}
\begin{proof}

{Assume we have a $\Box$-founded Magari algebra $\mathcal{A}=(A, \wedge, \vee, \to, 0, 1, \Box)$ and a filter $F$ of $\mathcal{A}$ such that $F\subset \Box^{-1} 1$. Notice that $F$ is an open filter, i.e. $\Box a \in F$ whenever $a\in F$. Indeed, if $a\in F$, then $\Box a =1$ and $\Box a \in F$. Hence, the quotient Magari algebra $\mathcal{A}/ F$ and the canonical epimorphism $f\colon \mathcal{A}\to\mathcal{A} / F$ are well-defined. 

Now we check that the algebra $\mathcal{A}/ F$ is $\Box$-founded. Assume there exists a sequence $(a_i)_{i\in \mathbb{N}}$ of elements of $\mathcal{A}$ such that $\Box_0 f(a_{i+1})\leqslant f(a_{i})$. We see that $f(\Box_0 a_{i+1} \to a_i)=1$ and $(\Box_0 a_{i+1} \rightarrow a_i) \in F$.   
Since $\Box_0 b=1$ for any $b \in F$, we have $\Box_0 \Box_0 a_{i+1} \leqslant  \Box_0 a_i$ in $\mathcal{A}$. From $\Box$-foundedness of $\mathcal{A}$, we obtain $\Box_0 a_i =1$ for any $i\in \mathbb{N}$. Since $(\Box_0 a_{i+1} \rightarrow a_i) \in F$, we also have $a_i \in F$ for all $i\in \mathbb{N}$. Consequently, $ f (a_0) =1$. The algebra $\mathcal{A}/ F$ is $\Box$-founded.

Let $\mathit{ht}(\mathcal{A}/F) \coloneq \sup \{ \mathit{ht}_{\mathcal{A}/F} (f(a))+1 \mid a \in A\setminus F \}$ and $F(\gamma)\coloneq f^{-1}(\mathit{M}_{\mathcal{A}/F}(\gamma))$. We see that $F(\gamma)$ is a filter of $\mathcal{A}$ for any ordinal $\gamma$ and $F(\mathit{ht}(\mathcal{A}/F))= F$. 
For an ultrafilter $u$ of $\mathcal{A}$, we set 
\[
\mathit{rk}(u):= 
\begin{cases}
 \min \{ \eta \dotminus 1 \mid \eta\leqslant \mathit{ht}(\mathcal{A}/F) \text{ and } F(\eta) \subset u\} & \text{if $F\subset u$};  \\
 \mathit{ht}(\mathcal{A}/F) & \text{otherwise}.
\end{cases}\]
In this definition, for any ordinal $\gamma$, $0\dotminus 1\coloneq 0$, $(\gamma+1)\dotminus 1 \coloneq \gamma$ and $\gamma\dotminus 1\coloneq \gamma$ if $\gamma$ is a limit ordinal. In addition, we put $\mathit{I}(\gamma) := \{u \in \mathit{Ult}\: \mathcal{A} \mid \mathit{rk}(u) < \gamma\}$. 
}

Set $\tau = \{ V \subset \mathit{Ult}\: \mathcal{A} \mid \forall u\in V \; \exists a \in A \;\; (\Box a \in u ) \wedge (\widehat{\boxdot a} \cap \mathit{I}(\mathit{rk}(u))\subset V) \}$,
where $\boxdot a = a \wedge \Box a$.

Let us check that $\tau$ is a topology on $\mathit{Ult}\: \mathcal{A}$. Trivially, $\emptyset \in \tau$ and $\tau$ is closed under arbitrary unions. For any $u \in  \mathit{Ult}\: \mathcal{A}$, we see that $\Box 1 = 1 \in u$ and $\widehat{\boxdot 1} \cap \mathit{I}(\mathit{rk}(u)) \subset \mathit{Ult}\:  \mathcal{A}$. Consequently $\mathit{Ult}\: \mathcal{A} \in \tau$. Assume $S_0 \in \tau$ and $S_1\in \tau$. Consider an arbitrary $u \in S_0 \cap S_1$. By definition of $\tau$, there exist elements $ b$ and $ c$ of $A$ such that $\Box b \in u$, $\Box c \in u$, $\widehat{\boxdot b} \cap \mathit{I}(\mathit{rk}(u))\subset S_0$ and $\widehat{\boxdot c} \cap \mathit{I}(\mathit{rk}(u))\subset S_1$.
We have $\Box (b \wedge c)=(\Box b \wedge \Box c) \in u$ and $\widehat{\boxdot (b\wedge c)} \cap \mathit{I}(\mathit{rk}(u)) = \widehat{\boxdot a} \cap \widehat{\boxdot c} \cap \mathit{I}(\mathit{rk}(u)) \subset S_0 \cap S_1$. Therefore $S_0 \cap S_1 \in \tau$. This shows that $\tau$ is a topology on $\mathit{Ult}\: \mathcal{A}$. 

It easily follows from the definition of $\tau$ that $\widehat{\boxdot a}\in \tau$, for any $a\in A$, and $ \mathit{I}(\gamma)\in \tau$, for any ordinal $\gamma$. Now we claim that $\tau$ is scattered. Consider any non-empty subset $S$ of  $\mathit{Ult}\: \mathcal{A}$. There is an ultrafilter $h\in S$ such that $\mathit{rk}(h) =\min \{\mathit{rk}(u) \mid u \in S\}$. We see that a set $\{h\} \cup \mathit{I}(\mathit{rk}(h))$ is a $\tau$-neighbourhood of $h$ and $S \cap \left(\{h\} \cup \mathit{I}(\mathit{rk}(h))\right) =\{h\}$. Hence the ultrafilter $h$ is an isolated point in $S$. This proves that $\tau$ is a scattered topology.    

Let us show that $\widehat{\Box a} = \mathit{cd}_\tau (\widehat{a})$ for any $a \in A$. First, we check that  $\widehat{\Box a} \subset \mathit{cd}_\tau (\widehat{a})$. For any ultrafilter $u$, if $u \in \widehat{\Box a}$, then $\widehat{\boxdot a} \cap \mathit{I}(\mathit{rk}(u))$ is a punctured neighbourhood of $u$. Also, $\widehat{\boxdot a} \cap \mathit{I}(\mathit{rk}(u)) \subset \widehat{a}$. By definition of the co-derived-set operator, $u \in  \mathit{cd}_\tau (\widehat{a})$. Consequently $\widehat{\Box a} \subset \mathit{cd}_\tau (\widehat{a})$. 

Now we claim that $\mathit{cd}_\tau (\widehat{a}) \subset\widehat{\Box a}$. Consider any ultrafilter $u$ such that $u \nin \widehat{\Box a}$. Let $W$ be an arbitrary punctured neighbourhood of $u$. It is sufficient to show that $W$ is not included in $\widehat{a}$. 


{By the definition of $\tau$, there exists an element $e$ of $A$ such that $\Box e \in u$ and $\widehat{\boxdot e} \cap \mathit{I}(\mathit{rk}(u)) \subset W$. From the conditions $\Box e \in u$ and $\Box a \nin u$, it follows that $\Box (\boxdot e \to a) \nin u$. Note that $\Box (\boxdot e \to a) \neq 1$, $ (\boxdot e \to a) \nin \Box^{-1} 1$, $ (\boxdot e \to a) \nin F$ and $\mathit{ht}_{\mathcal{A}/F}(\boxdot f(e) \to f(a))\neq \infty$. 

Let us check that $\mathit{ht}_{\mathcal{A}/F}(\boxdot f(e) \to f(a))< \mathit{rk}(u)$. 
If $F\subset u$, then $\Box (\boxdot e \to a) \nin F(\mathit{rk}(u)+1) \subset u$ and $\mathit{ht}_{\mathcal{A}/F}(\Box (\boxdot f(e) \to f(a)))\leqslant \mathit{rk}(u)$. From Lemma \ref{basic}, we have $\mathit{ht}_{\mathcal{A}/F}(\boxdot f(e) \to f(a))+1 \leqslant \mathit{ht}_{\mathcal{A}/F}(\Box (\boxdot f(e) \to f(a)))\leqslant \mathit{rk}(u)$. Hence, $\mathit{ht}_{\mathcal{A}/F}(\boxdot f(e) \to f(a))< \mathit{rk}(u)$. 
If $F\nsubset u$, then $\mathit{ht}_{\mathcal{A}/F}(\boxdot f(e) \to f(a))< \mathit{ht}(\mathcal{A}/F)= \mathit{rk}(u)$ since $ (\boxdot e \to a) \nin F$. Consequently, the inequality 
\[\mathit{ht}_{\mathcal{A}/F}(\boxdot f(e) \to f(a))< \mathit{rk}(u)\] 
holds in both cases. 

Recall that $\mathit{ht}_{\mathcal{A}/F}(\boxdot f(e) \to f(a))\neq \infty$. Therefore, 
\[(\boxdot e \to a) \nin F(\mathit{ht}_{\mathcal{A}/F}(\boxdot f(e) \to f(a))+1).\]
By the Boolean ultrafilter theorem, there exists an ultrafilter $h$ of $\mathcal{A}$ such that $(\boxdot e \to a) \nin h$ and $F(\mathit{ht}_{\mathcal{A}/F}(\boxdot f(e) \to f(a))+1)  \subset h$. We see that $\boxdot e \in h$, $a \nin h$, $F\subset h$ and $\mathit{rk}(h)\leqslant \mathit{ht}_{\mathcal{A}/F}(\boxdot f(e) \to f(a))$. We also have $\mathit{ht}_{\mathcal{A}/F}(\boxdot f(e) \to f(a))< \mathit{rk}(u)$.
It follows that $\mathit{rk}(h) < \mathit{rk}(u)$, $h \in \widehat{\boxdot e} \cap \mathit{I}(\mathit{rk}(u))$ and $h\nin\widehat{a}$. Consequently, $h$ is an element of $W$, which does not belong to $\widehat{a}$.
}

We obtain that none of the punctured neighbourhoods of $u$ are included in $\widehat{a}$. In other words, $u \nin \mathit{cd}_\tau (\widehat{a})$ for any $u \nin \widehat{\Box a}$. We establish that $\mathit{cd}_\tau (\widehat{a})\subset \widehat{\Box a} $. Hence, $\widehat{\Box a} = \mathit{cd}_\tau (\widehat{a})$.

{It remains to show that $\mathit{cd}_\tau (\bigcap \{\widehat{a}\mid a\in F\})= \mathit{Ult}\: \mathcal{A}$. Note that $\mathit{I}(\mathit{rk}(d))$ is a punctured neighbourhood of $d$ for any $d\in \mathit{Ult}\: \mathcal{A}$. Since $\mathit{rk}(d) \leqslant \mathit{ht}(\mathcal{A}/F)$, we have  
\[\mathit{I}(\mathit{rk}(d)) \subset \mathit{I}( \mathit{ht}(\mathcal{A}/F)) \subset \{u\in \mathit{Ult}\: \mathcal{A}\mid F\subset u\}= \bigcap \{\widehat{a}\mid a\in F\}.\]
Therefore, any ultrafilter has a punctured neighbourhood that is included in $ \bigcap \{\widehat{a}\mid a\in F\}$. Consequently, $\mathit{cd}_\tau (\bigcap \{\widehat{a}\mid a\in F\})= \mathit{Ult}\: \mathcal{A}$.  }

\end{proof}

\begin{thm}\label{representation of Magari algebras}
A Magari algebra is $\Box$-founded if and only if it is embeddable into the powerset Magari algebra of an Esakia frame.
\end{thm}
\begin{proof}
(if) Suppose a Magari algebra $\mathcal{A}$ is isomorphic to a subalgebra of the powerset Magari algebra of an Esakia frame $\mathcal{X}$. The powerset Magari algebra of $\mathcal{X}$ is $\sigma$-complete. Hence, by Proposition \ref{from sigma-completness to box-foundness}, it is $\Box$-founded. Since any subalgebra of a $\Box$-founded Magari algebra is $\Box$-founded, the algebra $\mathcal{A}$ is $\Box$-founded.

(only if) Suppose a Magari algebra $\mathcal{A}$ is $\Box$-founded. By Lemma \ref{important lemma}, {for the filter $F=\{1\}$}, there exists a scattered topology $\tau$ on $\mathit{Ult}\: \mathcal{A}$ such that $\widehat{\Box a} = \mathit{cd}_\tau (\widehat{a})$ for any element $a$ of $\mathcal{A}$. We know that $\mathcal{X}= (\mathit{Ult}\: \mathcal{A}, \mathit{cd}_\tau)$ is an Esakia frame by Proposition \ref{from scattered spaces to Esakia frames}. We see that the mapping $\:\widehat{\cdot}\;\colon a \mapsto \widehat{a}\:$ is an injective homomorphism from $\mathcal{A}$ to the powerset Magari algebra of the frame $\mathcal{X}$. Therefore the algebra $\mathcal{A}$ is embeddable into the powerset Magari algebra of an Esakia frame.
\end{proof}

For a Magari algebra $\mathcal{A}$, by $\mathit{Top} \:\mathcal{A}$, we denote the set of all scattered topologies $\tau $ on $\mathit{Ult}\: \mathcal{A}$ such that 
$\widehat{\Box a} = \mathit{cd}_\tau (\widehat{a})$ for any element $a$ of $\mathcal{A}$.
\begin{lem}\label{maximal extensions lemma}
Suppose $\mathcal{A}$ is a Magari algebra and $\tau \in \mathit{Top}\: \mathcal{A}$. Then there is a maximal with respect to inclusion element of $\mathit{Top}\: \mathcal{A}$ that extends $\tau$. 
\end{lem}
\begin{proof}
Consider the set $P= \{ \sigma \in \mathit{Top}\: \mathcal{A} \mid \tau \subset \sigma\}$, which is a partially ordered set with respect to inclusion. We claim that any chain in $P$ has an upper bound. 

Assume $C$ is a chain in $P$. Let $\nu$ be the coarsest topology containing $\tau$ and $\bigcup C$. Note that the topology $\nu$ is scattered as an extension of a scattered topology. For any element $a$ of $\mathcal{A}$, we have $\widehat{\Box a} = \mathit{cd}_\tau (\widehat{a})\subset  \mathit{cd}_\nu (\widehat{a})$, because $\nu$ is an extension of $\tau$. 

Now assume $c$ is an arbitrary element of $\mathcal{A}$ and $u\in \mathit{cd}_\nu (\widehat{c})$. We check that $u \in \widehat{\Box c}$. By definition of the co-derived-set operator, there is a punctured $\nu$-neighbourhood $V$ of $u$ such that $V \subset \widehat{c}$.
Since the set $\tau \cup \bigcup C$ is closed under finite intersections, it is a basis of $\nu$. Consequently there is a subset $W$ of $V$ with $ W \cup \{ u\}  \in \tau \cup \bigcup C$. We see that $W\subset \widehat{c}\;$ and $W$ is a punctured neighbourhood of $u$ with respect to a topology $\kappa \in \{\tau\} \cup C \subset \mathit{Top}\: \mathcal{A}$. Hence $u \in \mathit{cd}_\kappa (\widehat{c})= \widehat{\Box c}$.

We obtain that $\widehat{\Box a} = \mathit{cd}_\nu (\widehat{a})$ for any element $a$ of $\mathcal{A}$. Therefore $\nu \in \mathit{Top}\: \mathcal{A}$ and $\nu$ is an upper bound for $C$ in $P$. 

We see that any chain in $P$ has an upper bound. By Zorn's lemma, there is a maximal element in $P$, which is the required maximal extension of $\tau$. 
\end{proof}

The following lemma was inspired by Lemma 4.5 from \cite{Beklemishev2013}.
\begin{lem}\label{maximality property}
Suppose $\mathcal{A}$ is a Magari algebra and $\tau$ is a maximal element of $\mathit{Top}\: \mathcal{A}$. Then, for any $u \in \mathit{Ult}\: \mathcal{A}$ and any $V \in \tau$, we have $V\cup \{ u\}\in \tau$ or there are a $\tau$-open set $W$ and an element $a$ of $\mathcal{A}$ such that $u \in W$, $\Box a \nin u $ and $V\cap W \subset \widehat{a}$.
\end{lem}

\begin{proof}
Assume $u \in \mathit{Ult}\: \mathcal{A}$ and $V \in \tau$. It is sufficient to consider the case when $V\cup \{ u\}\nin \tau$. Let $\sigma$ be the coarsest topology containing $\tau$ and the set $V\cup \{ u\}$. The topology $\sigma$ is scattered as an extension of a scattered topology. 
Since $\tau$ is a maximal element of $\mathit{Top}\: \mathcal{A}$, the topology $\sigma$ does not belong to $\mathit{Top}\: \mathcal{A}$ and there exists an element $a$ of $\mathcal{A}$ such that $\widehat{\Box a} \neq \mathit{cd}_\sigma (\widehat{a})$. 
Notice that $\widehat{\Box a}= \mathit{cd}_\tau (\widehat{a})\subset \mathit{cd}_\sigma (\widehat{a})$, because $\tau \subset \sigma$. 
Thus there is an ultrafilter $h $ such that $h \in \mathit{cd}_\sigma (\widehat{a})$ and $h \nin \mathit{cd}_\tau (\widehat{a})=\widehat{\Box a} $. Hence there is a punctured $\sigma$-neighbourhood of $h$ that is included in $\widehat{a}$. In addition, note that $\tau \cup \{W \cap (V \cup \{u\}) \mid W \in \tau\} $ is a basis of $\sigma$. 
We see that $h\in B$ and $B \setminus \{ h\}  \subset \widehat{a}$ for some $B \in \tau \cup \{W \cap (V \cup \{u\}) \mid W \in \tau\}$.
If $B \in \tau$, then $h \in \mathit{cd}_\tau (\widehat{a})$. This is a contradiction with the condition $h \nin \mathit{cd}_\tau (\widehat{a})$. Therefore $B $ has the form $W \cap (V \cup \{u\})$ for some $W\in \tau$. Since $h\in B = W \cap (V \cup \{u\})$, we have $h \in V$ or $h=u$. 
If $h\in V$, then $h \in W \cap V$ and $(W \cap V)\setminus \{ h\} \subset \widehat{a}$. In this case, we obtain $h \in \mathit{cd}_\tau (\widehat{a})$, which is a contradiction. Consequently $h\nin V$ and $h=u$. It follows that $\Box a \nin u$, $u\in W$ and $W \cap V = (W \cap (V \cup \{u\}))\setminus \{ h\} \subset \widehat{a}$.

\end{proof}

For a scattered topological space $(X, \tau)$, the \emph{derivative topology $\tau^+$ on $X$} is defined as the coarsest topology including $\tau$ and $\{ \mathit{d}_\tau (Y) \mid Y \subset X\}$. The next lemma was inspired by Lemma 5.1 from \cite{Beklemishev2013}.
\begin{lem}\label{derived topology of a maximal one}
Suppose $\mathcal{A}=(A, \wedge, \vee, \to, 0, 1, \Box)$ is a Magari algebra and $\tau$ is a maximal element of $\mathit{Top}\: \mathcal{A}$. Then the topology $\tau^+$ is generated by $\tau$ and the sets $  \mathit{d}_\tau (\widehat{a})$ for $a \in A$.
\end{lem}

\begin{proof}
Assume $\tau$ is a maximal element of $\mathit{Top}\: \mathcal{A}$. Let $\tau^\prime$ be the topology generated by $\tau$ and the sets $ \mathit{d}_\tau (\widehat{a})$ for $a \in A$. It is clear that $\tau^\prime \subset \tau^+$. We prove the converse. We shall check that $ \mathit{d}_\tau (Y) $ is $\tau^\prime$-open for any $Y \subset \mathit{Ult}\: \mathcal{A}$.

Consider any $Y \subset \mathit{Ult}\: \mathcal{A}$ and any $u \in   \mathit{d}_\tau (Y)$. We claim that there is a $\tau^\prime$-neighbourhood of $u$ entirely contained in $\mathit{d}_\tau (Y)$. For any $X \subset \mathit{Ult}\: \mathcal{A}$, we denote the $\tau$-interior of $X$ by $\mathit{int}_\tau (X)$. Since $\tau$ is scattered, $\mathit{cd}_\tau (X)= \mathit{cd}_\tau (\mathit{int}_\tau (X))$. Put $X = \mathit{Ult}\: \mathcal{A} \setminus Y$. Since $u \in \mathit{d}_\tau (Y)$ and $u \nin \mathit{cd}_\tau (X)= \mathit{cd}_\tau (\mathit{int}_\tau (X))$, the set $\{u\} \cup \mathit{int}_\tau(X)  \nin \tau$. By Lemma \ref{maximality property}, there are a $\tau$-open set $W$ and an element $c$ of $ A$ such that $u \in W$, $\Box c \nin u $ and $\mathit{int}_\tau(X)\cap W \subset \widehat{c}$. It follows that 
\[u \in W \cap (\mathit{Ult}\: \mathcal{A} \setminus \widehat{\Box  c}) = W \cap  \mathit{d}_\tau(\widehat{\neg  c}) \in \tau^\prime.\] 
Thus $W \cap (\mathit{Ult}\: \mathcal{A} \setminus \widehat{\Box  c})$ is a $\tau^\prime$-neighbourhood of $u$. It remains to show that 
\[W \cap (\mathit{Ult}\: \mathcal{A} \setminus \widehat{\Box  c}) \subset \mathit{d}_\tau(Y).\]
Indeed, we have
\begin{multline}\label{ineq1}
\mathit{cd}_\tau(X) \cap W \subset \mathit{cd}_\tau( \mathit{int}_\tau(X)) \cap \mathit{cd}_\tau(W)=\\
=\mathit{cd}_\tau( \mathit{int}_\tau(X) \cap W) \subset \mathit{cd}_\tau( \widehat{ c})=  \widehat{\Box  c},
\end{multline} 
because $W$ is a $\tau$-open set and $\mathit{int}_\tau(X)\cap W \subset \widehat{c}$.
Hence,
\begin{align*}
W \cap (\mathit{Ult}\: \mathcal{A} \setminus \widehat{\Box  c}) &\subset W \cap (\mathit{Ult}\: \mathcal{A} \setminus (\mathit{cd}_\tau(X) \cap W)) \;\;\text{ (from (\ref{ineq1}))}  \\
&= W \cap ((\mathit{Ult}\: \mathcal{A} \setminus \mathit{cd}_\tau(X)) \cup ( \mathit{Ult}\: \mathcal{A} \setminus W))\\
&= W \cap ( \mathit{d}_\tau(Y) \cup ( \mathit{Ult}\: \mathcal{A} \setminus W)) \\
&= (W \cap  \mathit{d}_\tau(Y))\cup (W \cap (\mathit{Ult}\: \mathcal{A} \setminus W))\\
&= W \cap  \mathit{d}_\tau(Y)\\
&\subset\mathit{d}_\tau(Y).
\end{align*} 

This argument shows that any element of $\mathit{d}_\tau(Y)$ belongs to this set together with a $\tau^\prime$-neighbourhood. We conclude that $ \mathit{d}_\tau (Y) $ is $\tau^\prime$-open and $\tau^\prime = \tau^+$.

\end{proof}

\section{Neighbourhood completeness}
\label{s5}
In this section, we prove neighbourhood completeness of $\mathsf{GLP}$ extended with infinitary derivations. We also show that any $\Box$-founded $\mathsf{GLP}$-algebra can be embedded into the powerset algebra of a $\mathsf{GLP}$-frame. 

Analogously to the case of Magari algebras, by $\mathit{Ult}\: \mathcal{A}$, we denote the set of ultrafilters of a $\mathsf{GLP}$-algebra $\mathcal{A}$.
For a $\mathsf{GLP}$-algebra $\mathcal{A}= (A, \wedge, \vee, \to, 0, 1, \Box_0, \Box_1, \dotsc)$, we denote the Magari algebra $(A, \wedge, \vee, \to, 0, 1, \Box_i)$ by $\mathcal{A}_i$. 
We see $\mathit{Ult}\: \mathcal{A}=\mathit{Ult}\: \mathcal{A}_i$ for any $i \in \mathbb{N}$. 
We call (maximal with respect to inclusion) elements of $\mathit{Top}\: \mathcal{A}_i$ \emph{(maximal) $i$-topologies} on $\mathit{Ult}\: \mathcal{A}$.

\begin{lem}\label{recursive step}
For any $\mathsf{GLP}$-algebra $\mathcal{A}$ and any maximal $i$-topology $\tau$ on $\mathit{Ult}\: \mathcal{A}$, there exists a maximal $(i+1)$-topology $\nu$ on $\mathit{Ult}\: \mathcal{A}$ such that $\tau \subset \nu$ and $\mathit{d}_{\tau}(Y)$ is $\nu$-open for each $Y\subset \mathit{Ult}\: \mathcal{A}$.
\end{lem}
\begin{proof}
Assume we have a $\mathsf{GLP}$-algebra $\mathcal{A}$ and a maximal $i$-topology $\tau$ on $\mathit{Ult}\: \mathcal{A}$. Consider the coarsest topology $\tau^\prime$ containing $\tau^+$ and all sets of the form $ \{u\} \cup \widehat{\boxdot_{i+1} a}$, where $  u \in  \mathit{Ult}\: \mathcal{A}$, $\Box_{i+1} a \in u $ and $\boxdot_{i+1} a = a \wedge \Box_{i+1} a$. We see that 
$\tau \subset \tau^\prime$ and $\mathit{d}_{\tau}(Y)$ is $\tau^\prime$-open for each $Y\subset \mathit{Ult}\: \mathcal{A}$. Trivially, the topology $\tau^\prime$ is scattered as an extension of a scattered topology. We claim that $\tau^\prime \in \mathit{Top}\: \mathcal{A}_{i+1}$. 

We shall show that $\widehat{\Box_{i+1} a} = \mathit{cd}_{\tau^\prime} (\widehat{a})$ for any element $a$ of $\mathcal{A}$. First, we check that  $\widehat{\Box_{i+1} a} \subset \mathit{cd}_{\tau^\prime} (\widehat{a})$. For any ultrafilter $d$, if $d \in \widehat{\Box_{i+1} a}$, then $\widehat{\boxdot_{i+1} a}$ is a punctured $\tau^\prime$-neighbourhood of $d$. Also, $\widehat{\boxdot_{i+1} a} \subset \widehat{a}$. By definition of the co-derived-set operator, $d \in  \mathit{cd}_{\tau^\prime} (\widehat{a})$. Consequently $\widehat{\Box_{i+1} a} \subset \mathit{cd}_{\tau^\prime} (\widehat{a})$.

Now we check that $\mathit{cd}_{\tau^\prime} (\widehat{a}) \subset\widehat{\Box_{i+1} a}$. Consider any ultrafilter $d$ such that $d \nin \widehat{\Box_{i+1} a}$. In addition, let $W$ be an arbitrary punctured $\tau^\prime$- neighbourhood of $d$. It is sufficient to show that $W$ is not included in $\widehat{a}$. 

We have $\Box_{i+1} a \nin d$, $d \nin W$ and $W \cup \{d\} \in \tau^\prime$. From Lemma \ref{derived topology of a maximal one}, there is a basis of $\tau^\prime$ consisting of alls sets of the form 
\[V \cap \mathit{d}_{\tau} (\widehat{b_1}) \cap \dotsb \cap \mathit{d}_{\tau} (\widehat{b_n}) \cap  ( \{u_1\} \cup \widehat{\boxdot_{i+1} c_1}) \cap \dotsb   \cap  ( \{u_m\} \cup \widehat{\boxdot_{i+1} c_m}),\]
where $V \in \tau$, $\{ b_1, \dotsc, b_n\} $ and $\{ c_1, \dotsc, c_m\} $ are (possibly empty) subsets of $A$, $\{ u_1, \dotsc, u_m\} $ is a subset of $\mathit{Ult}\: \mathcal{A}$. In addition, $\Box_{i+1} c_k \in u_k$ for $k \in \{ 1, \dotsc, m\} $. Hence we have
\[d \in \left(V \cap \mathit{d}_{\tau} (\widehat{b_1}) \cap \dotsb \cap \mathit{d}_{\tau} (\widehat{b_n}) \cap  ( \{u_1\} \cup \widehat{\boxdot_{i+1} c_1}) \cap \dotsb   \cap  ( \{u_m\} \cup \widehat{\boxdot_{i+1} c_m}) \right)\subset W \cup \{d\}\]  
for some element of the basis of $\tau^\prime$.
We see that the ultrafilter $d$ contains $\Diamond_i b_1, \dotsc , \Diamond_i b_n$ and  $\Box_{i+1} c_1, \dotsc , \Box_{i+1} c_m$. Also, $\Diamond_{i+1} \neg a \in d$. 
In any $\mathsf{GLP}$-algebra, we have
\begin{gather*}
\bigwedge \{\Diamond_i b_1, \dotsc , \Diamond_i b_n\} \leqslant \Box_{i+1} \bigwedge \{\Diamond_i b_1, \dotsc , \Diamond_i b_n\},\\
\bigwedge \{\Box_{i+1} c_1, \dotsc , \Box_{i+1} c_m \} \leqslant \Box_{i+1} \bigwedge \{\boxdot_{i+1} c_1, \dotsc , \boxdot_{i+1} c_m \}. 
\end{gather*}
Further, we have
\begin{multline*}
(\Diamond_{i+1} \neg a) \wedge \bigwedge \{\Diamond_i b_1, \dotsc , \Diamond_i b_n, \Box_{i+1} c_1, \dotsc , \Box_{i+1} c_m \} \leqslant \\ \leqslant
(\Diamond_{i+1} \neg a) \wedge \Box_{i+1} \bigwedge \{\Diamond_i b_1, \dotsc , \Diamond_i b_n\} \wedge \Box_{i+1} \bigwedge \{\boxdot_{i+1} c_1, \dotsc , \boxdot_{i+1} c_m \} \leqslant\\
\leqslant \Diamond_{i+1} \left((\neg a) \wedge \bigwedge \{\Diamond_i b_1, \dotsc , \Diamond_i b_n, \boxdot_{i+1} c_1, \dotsc , \boxdot_{i+1} c_m \} \right) \leqslant\\
\leqslant \Diamond_{i} \left((\neg a) \wedge \bigwedge \{\Diamond_i b_1, \dotsc , \Diamond_i b_n, \boxdot_{i+1} c_1, \dotsc , \boxdot_{i+1} c_m \} \right)
\end{multline*}
We obtain $\Diamond_{i} \left((\neg a) \wedge \bigwedge \{\Diamond_i b_1, \dotsc , \Diamond_i b_n, \boxdot_{i+1} c_1, \dotsc , \boxdot_{i+1} c_m \} \right) \in d$ and
\[d \in \mathit{d}_{\tau}  \left( \widehat{\neg a}\cap \mathit{d}_{\tau} (\widehat{b_1}) \cap \dotsb \cap \mathit{d}_{\tau} (\widehat{b_n}) \cap \widehat{\boxdot_{i+1} c_1} \cap \dotsb   \cap  \widehat{\boxdot_{i+1} c_m}\right).\]
Since $V$ is a $\tau$-neighbourhood of $d$, there exists an ultrafilter $w$ such that
\[w \in (V \setminus \{d\}) \cap \widehat{\neg a}\cap \mathit{d}_{\tau} (\widehat{b_1}) \cap \dotsb \cap \mathit{d}_{\tau} (\widehat{b_n}) \cap \widehat{\boxdot_{i+1} c_1} \cap \dotsb   \cap  \widehat{\boxdot_{i+1} c_m}\subset W.\]
Consequently $w$ is an element of $W$, which does not belong to $\widehat{a}$.

We obtain that none of the punctured $\tau^\prime$-neighbourhoods of $d$ are included in $\widehat{a}$. In other words, $d \nin \mathit{cd}_{\tau^\prime} (\widehat{a})$ for any $d \nin \widehat{\Box_{i+1} a}$. This argument shows that $\mathit{cd}_{\tau^\prime} (\widehat{a})\subset \widehat{\Box_{i+1} a} $. Hence $\widehat{\Box_{i+1} a} = \mathit{cd}_{\tau^\prime} (\widehat{a})$. We see $\tau^\prime \in \mathit{Top}\: \mathcal{A}_{i+1}$.

Now we extend the topology $\tau^\prime$ applying Lemma \ref{maximal extensions lemma} and obtain the required maximal $(i+1)$-topology $\nu$ on $\mathit{Ult}\: \mathcal{A}$.

\end{proof}

\begin{lem}\label{important lemma 2}
Suppose $\mathcal{A}= (A, \wedge, \vee, \to, 0, 1, \Box_0, \Box_1, \dotsc)$ is a $\Box$-founded $\mathsf{GLP}$-algebra and $F$ is a filter of $\mathcal{A}$ such that $F\subset \Box^{-1}_0 1$, where $\Box^{-1}_0 1= \{a\in A \mid \Box_0 a =1\}$. Then there exists a series of topologies $\tau_0, \tau_1, \dotsc$ on $\mathit{Ult}\: \mathcal{A}$ such that $(\mathit{Ult}\: \mathcal{A}, \tau_0, \tau_1, \dotsc)$ is a $\mathsf{GLP}$-space and $\tau_i \in \mathit{Top}\: \mathcal{A}_i$ for every $i\in \mathbb{N}$. Moreover, $\mathit{cd}_{\tau_0} (\bigcap \{\widehat{a}\mid a\in F\})= \mathit{Ult}\: \mathcal{A}$.
\end{lem}
\begin{proof}
From Lemma \ref{important lemma}, there exists a topology $\tau \in  \mathit{Top}\: \mathcal{A}_0$ such that $\mathit{cd}_{\tau} (\bigcap \{\widehat{a}\mid a\in F\})= \mathit{Ult}\: \mathcal{A}$. By Lemma \ref{maximal extensions lemma}, the topology $\tau$ can be extended to a maximal 0-topology $\tau_0$. Applying Lemma \ref{recursive step}, we obtain a series of topologies $\tau_1, \tau_2, \dotsc$ on $\mathit{Ult}\: \mathcal{A}$ such that $(\mathit{Ult}\: \mathcal{A}, \tau_0, \tau_1, \dotsc)$ is a $\mathsf{GLP}$-space and $\tau_i \in \mathit{Top}\: \mathcal{A}_i$ for any $i\in \mathbb{N}$. Since $\tau\subset \tau_0$, we also have $\mathit{cd}_{\tau_0} (\bigcap \{\widehat{a}\mid a\in F\})= \mathit{Ult}\: \mathcal{A}$.
\end{proof}

The following theorem is analogous to Theorem \ref{representation of Magari algebras} and is obtained by a similar argument. So we omit the proof.
\begin{thm}
A $\mathsf{GLP}$-algebra is $\Box$-founded if and only if it is embeddable into the powerset $\mathsf{GLP}$-algebra of a $\mathsf{GLP}$-frame.
\end{thm}

Now we establish neighbourhood completeness of $\mathsf{GLP}$ extended with infinitary derivations.

\begin{thm}[neighbourhood completeness]\label{neighbourhood completeness}
For any sets of formulas $\Sigma$ and $\Gamma$, and any formula $\varphi$, we have
\[\Sigma ;\Gamma \Vdash \varphi \Longleftrightarrow  \Sigma ;\Gamma \VDash \varphi \Longleftrightarrow \Sigma ; \Gamma \vDash \varphi.\]
\end{thm}
\begin{proof}
From Theorem \ref{algebraic completeness} and Proposition \ref{neighbourhood soundness}, it remains to show that $\Sigma ;\Gamma \VDash \varphi$ whenever $ \Sigma ;\Gamma \vDash \varphi$. Assume $ \Sigma ;\Gamma \vDash \varphi$. Also, assume we have a $\Box$-founded $\mathsf{GLP}$-algebra $\mathcal{A}= (A, \wedge, \vee, \to, 0, 1, \Box_0, \Box_1, \dotsc)$ and a valuation $v$ in $\mathcal{A}$ such that $ \Box_0 v(\xi)=1 $ for any $ \xi \in \Sigma$. We prove $v(\varphi)\in \langle \{v(\psi)\mid \psi \in \Gamma\}\rangle$ by \textit{reductio ad absurdum}.

If $v(\varphi)\nin \langle \{v(\psi)\mid \psi \in \Gamma\}\rangle$, then, by the Boolean ultrafilter theorem, there exists an ultrafilter $h$ of $\mathcal{A}$ such that $v(\varphi)\nin h$ and $\langle \{v(\psi)\mid \psi \in \Gamma\}\rangle \subset h$. In addition, let $F= \langle \{v(\xi)\mid \xi \in \Sigma\}\rangle$. Notice that $F\subset \Box^{-1}_0 1$. By Lemma \ref{important lemma 2}, there is a series of topologies $\tau_0, \tau_1, \dotsc$ on $\mathit{Ult}\: \mathcal{A}$ such that $(\mathit{Ult}\: \mathcal{A}, \tau_0, \tau_1, \dotsc)$ is a $\mathsf{GLP}$-space and $\tau_i \in \mathit{Top}\: \mathcal{A}_i$ for any $i\in \mathbb{N}$. Moreover, $\mathit{cd}_{\tau_0} (\bigcap \{\widehat{a}\mid a\in F\})= \mathit{Ult}\: \mathcal{A}$.
Note that $\mathcal{X}=(\mathit{Ult}\: \mathcal{A},  \mathit{cd}_{\tau_0}, \mathit{cd}_{\tau_1}, \dotsc)$ is a $\mathsf{GLP}$-frame by Proposition \ref{GLP-spaces=GLP-frames}. Besides, the mapping $\:\widehat{\cdot}\;\colon a \mapsto \widehat{a}\:$ is an embedding of the $\mathsf{GLP}$-algebra $\mathcal{A}$ into the powerset $\mathsf{GLP}$-algebra of $\mathcal{X}$. 
 
We see that $w \colon \theta \mapsto \widehat{v(\theta)}$ is valuation over $\mathcal{X}$. Since $\bigcap \{\widehat{a}\mid a\in F\}\subset w(\xi)$ for each $\xi\in \Sigma $ and $ \mathit{cd}_{\tau_0} (\bigcap \{\widehat{a}\mid a\in F\})= \mathit{Ult}\: \mathcal{A}$, there is a $0$-neighbourhood $U$ of $h$ in $\mathcal{X}$ such that 
\[\forall u \in U\setminus \{h\} \;\;\forall \xi \in \Sigma \;\; (\mathcal{X},w), u \vDash \xi . \]  
Since $\langle \{v(\psi)\mid \psi \in \Gamma\}\rangle \subset h$, we have $(\mathcal{X},w), h \vDash \psi $ for every $\psi \in \Gamma$. From the assumption $ \Sigma; \Gamma \vDash \varphi$, we also have $ \Sigma; \Gamma \vDash^\ast \varphi$ by Proposition \ref{two glocal neighbourhood relations}. Therefore, $(\mathcal{X},w),h \vDash \varphi $ and $v(\varphi)\in h$, which is a contradiction with the condition $v(\varphi)\nin h$. We conclude that $v(\varphi)\in \langle \{v(\psi)\mid \psi \in \Gamma\}\rangle$.

\end{proof}

\begin{cor}
For any set of formulas $\Gamma$ and any formula $\varphi$, we have
\[
\Gamma \Vdash_l \varphi \Longleftrightarrow  \Gamma \VDash_l \varphi \Longleftrightarrow  \Gamma \vDash_l \varphi,\quad \quad
\Gamma \Vdash_g \varphi \Longleftrightarrow  \Gamma \VDash_g \varphi \Longleftrightarrow  \Gamma \vDash_g \varphi.
\]
\end{cor}

\section{From infinitary derivations to ordinary ones}
\label{s6}

We now clarify the connection between the original $\mathsf{GLP}$ and its infinitary extension. To do so, we generalize the ideas of the bachelor’s thesis \cite{Chepasov2022} written under our supervision by Anatoliy Chepasov.

For the logic $\mathsf{GLP}$, words are usually defined as formulas of the form $\Diamond_{i_1} \dotso \Diamond _{i_k} \top$, where $k$ can be $0$. However, it is more convenient for us to define them in the dual way. We will understand \emph{words} as formulas of the form $\Box_{i_1} \dotso \Box _{i_k} \bot$. Now we recall some facts about words that will be needed in what follows.

\begin{thm}[{see Proposition 3 and Theorem 1 from \cite{Beklemishev2005b}}]\label{well-foundness for words}There is no sequence of words $(\varphi_i)_{i\in \mathbb{N}}$ such that $\mathsf{GLP}\vdash \Box \varphi_{i+1} \to \varphi_i$ for all $i\in \mathbb{N}$.
\end{thm}

\begin{lem}[{see Corollary 12 from \cite{Beklemishev2005b}}]\label{l1}
Every ground formula, i.e. every formula that do not contain variables, is provably equivalent in $\mathsf{GLP}$ to a Boolean combination of words.
\end{lem}

\begin{lem}[{see Lemma 9 from \cite{Beklemishev2005b}}]\label{l2}
A disjunction of two words $\varphi \vee \psi$ is provably equivalent in $\mathsf{GLP}$ to a word.
\end{lem}
\begin{lem}[{see Lemma 10 from \cite{Beklemishev2005b}}]\label{l3}
For any words $\alpha$, $\beta_1, \dotsc , \beta_k$, the formula 
\[\Box_0 (\bigwedge_{1\leqslant i\leqslant k} \beta_i \to \alpha)\]
is equivalent in $\mathsf{GLP}$ to the word $\Box_0 \alpha$ or to the formula $\top$.
\end{lem}
\begin{lem}[{see Corollary 6 from \cite{Beklemishev2005b}}]\label{l4}
For any words $\varphi$ and $\psi$, we have $\mathsf{GLP}\vdash \Box_0 \varphi \to \Box_0 \psi$ or $\mathsf{GLP}\vdash \Box_0 \psi \to \Box_0 \varphi$.
\end{lem}

\begin{prop}[{see Lemma 4.1.2 from \cite{Chepasov2022}}]\label{ground formulas to words}
For any ground formula $\varphi$, the formula $\Box_0 \varphi$ is provably equivalent in $\mathsf{GLP}$ to a word or to the formula $\top$.
\end{prop}
\begin{proof}
By Lemma \ref{l1}, the formula $\varphi$ is provably equivalent in $\mathsf{GLP}$ to a conjunction $\bigwedge \{\varphi_0, \dotsc, \varphi_k\}$, where every $\varphi_i$ has the form $\bigwedge\Gamma_i \to \bigvee \Delta_i$ for some finite sets of words $\Gamma_i$ and $\Delta_i$. From Lemma \ref{l2}, each formula $\bigvee\Delta_i$ is equivalent to a word $\alpha_i$, where $\bigvee \emptyset$ is the word $\bot$. Note that $ \mathsf{GLP}\vdash\Box_0 \varphi\leftrightarrow\bigwedge \{\Box_0\varphi_0, \dotsc, \Box_0\varphi_k\}$. By Lemma \ref{l3}, every conjunct $\Box_0 \varphi_i$ is equivalent to the formula $\top$ or to the word $\Box_0 \alpha_i$. If all conjuncts are equivalent to $\top$, then $\Box_0 \varphi$ is also equivalent to $\top$. Otherwise, $\Box_0 \varphi$ is equivalent to a finite non-empty conjunction of words of the form $\Box_0 \alpha_i$. From Lemma \ref{l4}, there exists $i$ such that $\mathsf{GLP}\vdash \Box_0 \varphi \leftrightarrow \Box_0 \alpha_i $, which concludes the proof.
\end{proof}


Let us remind the reader the uniform interpolation property of $\mathsf{GLP}$. By $\mathit{PV}(\varphi)$, we denote the set of propositional variables of $\varphi$. A formula $\theta$ is called the \emph{$p$-coprojection of a formula $\varphi$} if $\mathit{PV}(\theta)\subset \mathit{PV}(\varphi)\setminus\{p\}$ and, for any $\psi$ such that $p\nin \mathit{PV}(\psi)$, we have
\[\mathsf{GLP}\vdash \psi\to \varphi \Longleftrightarrow \mathsf{GLP}\vdash \psi\to \theta.\]
The latter condition can be equivalently stated as:
\begin{itemize}
\item $\mathsf{GLP}\vdash \theta \to \varphi$;
\item if $\mathsf{GLP}\vdash \psi\to \varphi$ and $p\nin \mathit{PV}(\psi)$, then $\mathsf{GLP}\vdash \psi\to \theta$.
\end{itemize}
\begin{thm}[{see Theorem 3.1 from \cite{Shamkanov2011}}]
The logic $\mathsf{GLP}$ has the uniform interpolation property, i.e. there exists a $p$-coprojection for any formula
and any propositional variable $p$.
\end{thm}

Note that a $p$-coprojection of a formula $\varphi$ is unique up to the provable equivalence in $\mathsf{GLP}$. Let us denote this $p$-coprojection of $\varphi$ by $\forall p\; \varphi$. For a formula $\varphi$ such that $\mathit{PV}(\varphi)=\{p_1, \dotsc, p_n\}$, we define its \emph{universal closure $\overline{\forall}\varphi$} as $\forall p_n \; \dotso \forall p_1 \;\varphi$. Trivially, $\overline{\forall}\varphi$ is a ground formula, i.e. $ \mathit{PV}(\overline{\forall}\varphi)=\emptyset$. Also, we have:
\begin{itemize}
\item $\mathsf{GLP}\vdash \overline{\forall}\varphi \to \varphi$;
\item if $\mathsf{GLP}\vdash \psi\to \varphi$ and $\psi$ is a ground formula, then $\mathsf{GLP}\vdash \psi\to \overline{\forall}\varphi$.
\end{itemize}

\begin{thm} \label{fin <-> inf for fin sigma}
For any finite set of formulas $\Sigma$, any set of formulas $\Gamma$, and any formula $\varphi$, we have
\[\Sigma ;\Gamma \vdash \varphi \Longleftrightarrow \Sigma ;\Gamma \Vdash \varphi.\footnote{Note that this result could also be obtained semantically similar to the proof of Proposition 2.2 given in the appendix in \cite{Shamkanov2020a}.}\]
\end{thm}
\begin{proof}
The left-to-right implication immediately holds. We obtain the converse. Assume $\Sigma ;\Gamma \Vdash \varphi$. We have an $\omega$-derivation $\delta$ with the root marked by $\varphi$ in which all boxed assumption leaves are marked by some elements of $\Sigma$ and all non-boxed assumption leaves are marked by some elements of $\Gamma$. 
By transfinite induction on the ordinal height of $\delta $, we prove that $\Sigma ;\Gamma \vdash \varphi$.

If $\varphi$ is an axiom of $\mathsf{GLP}$ or an element of $\Gamma$, then $\Sigma ;\Gamma \vdash \varphi$. Otherwise, consider the lowermost application of an inference rule in $\delta$. 

Case 1. Suppose that $\delta$ has the form
\begin{gather*}
\AXC{$\delta^\prime$}
\noLine
\UIC{$\vdots$}
\noLine
\UIC{$\psi$}
\AXC{$\delta^{\prime\prime}$}
\noLine
\UIC{$\vdots$}
\noLine
\UIC{$\psi \to \varphi$}
\LeftLabel{$\mathsf{mp}$}
\RightLabel{ .}
\BIC{$\varphi$}
\DisplayProof
\end{gather*}
By the induction hypothesis applied to $ \delta^\prime $ and $\delta^{\prime\prime} $, we have 
$\Sigma ;\Gamma \vdash \psi\to\varphi$ and $\Sigma ;\Gamma \vdash \psi$. Therefore, $\Sigma ;\Gamma \vdash \varphi$.

Case 2. Suppose that $\delta$ has the form
\begin{gather*}
\AXC{$\delta^\prime$}
\noLine
\UIC{$\vdots$}
\noLine
\UIC{$\psi$}
\LeftLabel{$\mathsf{nec}$}
\RightLabel{ ,}
\UIC{$\Box_0 \psi$}
\DisplayProof
\end{gather*}
where $\Box_0 \psi =\varphi$. We see that $\Sigma ; \Sigma \Vdash \psi$. 
By the induction hypothesis, we have $\Sigma ;\Sigma \vdash \psi$. Applying the rule ($\mathsf{nec}$), we obtain $\Sigma ; \emptyset \vdash \varphi$ and $\Sigma ; \Gamma \vdash \varphi$.

Case 3. Suppose that $\delta$ has the form
\begin{gather*}
\AXC{$\delta^\prime$}
\noLine
\UIC{$\vdots$}
\noLine
\UIC{$\Box_0\varphi_{1} \rightarrow \varphi_0$}
\AXC{$\delta^{\prime\prime}$}
\noLine
\UIC{$\vdots$}
\noLine
\UIC{$\Box_0\varphi_{2} \rightarrow \varphi_1$}
\AXC{$\delta^{\prime\prime\prime}$}
\noLine
\UIC{$\vdots$}
\noLine
\UIC{$\Box_0\varphi_{3} \rightarrow \varphi_2$}
\AXC{$\dots$}
\LeftLabel{$\omega$}
\RightLabel{ ,}
\QIC{$ \varphi_0$}
\DisplayProof 
\end{gather*}
where $\varphi_0=\varphi$.
We see that $\Sigma ; \Sigma \Vdash \Box_0\varphi_{n+2} \rightarrow \varphi_{n+1}$ for every $n\in \mathbb{N}$. By the induction hypothesis, we have $\Sigma ;\Sigma \vdash \Box_0\varphi_{n+2} \rightarrow \varphi_{n+1}$. Hence, for every $n\in \mathbb{N}$, $\mathsf{GLP} \vdash \bigwedge \Sigma \wedge \Box_0 \bigwedge \Sigma \to (\Box_0\varphi_{n+2} \rightarrow \varphi_{n+1})$ and $\mathsf{GLP} \vdash \Box_0\xi_{n+2} \rightarrow \xi_{n+1}$, where $\xi_n\coloneq \bigwedge \Sigma \wedge \Box_0 \bigwedge \Sigma \to \varphi_n$. Furthermore, $\mathsf{GLP} \vdash \Box_0 \overline{\forall}\xi_{n+2} \rightarrow \overline{\forall}  \xi_{n+1}$ and $\mathsf{GLP} \vdash \Box_0 \Box_0 \overline{\forall}\xi_{n+2} \rightarrow \Box_0 \overline{\forall}  \xi_{n+1}$. Applying Theorem \ref{well-foundness for words} and Proposition \ref{ground formulas to words}, we obtain $\mathsf{GLP} \vdash  \Box_0 \overline{\forall} \xi_{n+1}$ for some $n\in \mathbb{N}$. Therefore, $\mathsf{GLP} \vdash \Box_0 \overline{\forall}  \xi_{1}$, $\mathsf{GLP} \vdash \Box_0 \xi_{1}$ and $\mathsf{GLP} \vdash \Box_0( \bigwedge \Sigma \wedge \Box_0 \bigwedge \Sigma \to \varphi_1)$. Consequently, $\mathsf{GLP} \vdash  \Box_0 \bigwedge \Sigma \to \Box_0 \varphi_1$, $\Sigma ; \emptyset \vdash \Box_0\varphi_1$ and $\Sigma ; \Gamma \vdash \Box_0\varphi_1$. By the induction hypothesis applied to $ \delta^\prime $, we also have $\Sigma ;\Gamma  \vdash \Box_0\varphi_{1} \rightarrow \varphi_0$. Hence, we obtain $\Sigma ;\Gamma  \vdash  \varphi_0$, i.e. $\Sigma ;\Gamma  \vdash  \varphi$.
\end{proof}
\begin{cor}
The inference rule 
\[
\AXC{$\Box_0\varphi_{1} \rightarrow \varphi_0$}
\AXC{$\Box_0\varphi_{2} \rightarrow \varphi_1$}
\AXC{$\Box_0\varphi_{3} \rightarrow \varphi_2$}
\AXC{$\dots$}
\LeftLabel{$\omega$}
\RightLabel{ .}
\QIC{$ \varphi_0$}
\DisplayProof 
\] 
is admissible in $\mathsf{GLP}$, i.e. $\mathsf{GLP}\vdash \varphi_0$ whenever $\mathsf{GLP}\vdash \Box \varphi_{i+1} \to \varphi_i$ for all $i\in \mathbb{N}$.
\end{cor}
\begin{proof}
If $\mathsf{GLP}\vdash \Box \varphi_{i+1} \to \varphi_i$ for all $i\in \mathbb{N}$, then $\emptyset ; \emptyset \Vdash \varphi_0$. By the previous theorem, $\mathsf{GLP}  \vdash \varphi_0$ 
\end{proof}
Applying Theorem \ref{neighbourhood completeness} and Theorem \ref{fin <-> inf for fin sigma}, we immediately obtain strong local neighbourhood completeness of $\mathsf{GLP} $. Obviously, weak global neighbourhood completeness also holds.
\begin{cor}
We have
\begin{gather*}
\Gamma \vdash_l \varphi \Longleftrightarrow  \Gamma \vDash_l \varphi,\quad \quad
\Gamma \vdash_g \varphi  \Longleftrightarrow  \Gamma \vDash_g \varphi,
\end{gather*}
where the second equivalence holds for finite sets $\Gamma$.
\end{cor}
\begin{remark}
In the case of $\mathsf{GLP} $ with ordinary proofs, strong global neighbourhood completeness does not hold since $\{\Box_0 p_{i+1} \to p_i\mid i\in\mathbb{N}\}\vDash_g p_0$ and $\{\Box_0 p_{i+1} \to p_i\mid i\in\mathbb{N}\}\nvdash_g p_0$.  
\end{remark}

\paragraph*{Acknowledgements.} 

I am grateful to Alexander Citkin and Alexei Muravitsky for the opportunity to join this volume in memory of Alexander Vladimirovich Kuznetsov.
Heartily thanks to my beloved wife Mariya Shamkanova for her warm and constant support. S.D.G.

\bibliographystyle{amsplain}
\bibliography{Collected}

\end{document}